\documentclass[intlimits]{amsart}

\usepackage{amsmath}
\usepackage{a4wide} 
\usepackage{amssymb}
\usepackage[mathcal]{euscript}

\usepackage{mathrsfs}

\theoremstyle{plain}
\newtheorem{thm}{Theorem}[section]
\newtheorem{lem}[thm]{Lemma}
\newtheorem{prop}[thm]{Proposition}
\newtheorem{cor}[thm]{Corollary}

\theoremstyle{definition}
\newtheorem{df}[thm]{Definition}

\theoremstyle{remark}
\newtheorem{rem}[thm]{Remark}

\newcommand{\ds}{\mathrm{\,d}s}
\newcommand{\dt}{\mathrm{\,d}t}
\newcommand{\dx}{\mathrm{\,d}x}
\newcommand{\dy}{\mathrm{\,d}y}
\newcommand{\dz}{\mathrm{\,d}z}

\newcommand{\esssup}{\operatornamewithlimits{ess\, sup\,}}

\newcommand{\R}{\mathbb{R}}
\newcommand{\N}{\mathbb{N}}
\newcommand{\Z}{\mathbb{Z}}
\newcommand{\K}{\mathbb{K}}
\newcommand{\Kj}{\mathbb{K}_1}
\newcommand{\Kd}{\mathbb{K}_2}
\newcommand{\hra}{\hookrightarrow}
\newcommand{\M}{\mathscr M}
\newcommand{\MM}{\mathscr M_+}

\newcommand{\Lqw}{\Lambda^q(w)}

\newcommand{\Up}{U^p}
\newcommand{\Upm}{U^\frac{p}{m}}
\newcommand{\CL}{{C\!L^{m,p}(u,v)}}
\newcommand{\CLj}{{C\!L^{1,p}(u,v)}}

\let\plusminus\pm

\newcommand{\jq}{\frac1q}
\newcommand{\jp}{\frac1p}
\newcommand{\jr}{\frac1r}
\newcommand{\mjp}{{-\frac1p}}

\renewcommand{\rq}{{\frac rq}}
\newcommand{\qp}{{\frac qp}}
\newcommand{\qqp}{{\frac {q'}p}}
\newcommand{\qr}{{\frac qr}}

\renewcommand{\rq}{\frac{r}{q}}

\renewcommand{\pm}{\frac pm}
\renewcommand{\mp}{\frac mp}

\newcommand{\fii}{\varphi}

\newcommand{\sumk}{\sum_{k\in\K}}
\newcommand{\sumkj}{\sum_{k\in\Kj}}
\newcommand{\sumkd}{\sum_{k\in\Kd}}
\newcommand{\sumkmj}{\sum_{k\in\K-1}}
\newcommand{\sumkmd}{\sum_{k\in\K-2}}

\newcommand{\tj}{{t_j}}
\newcommand{\tk}{{t_k}}
\newcommand{\tjmj}{{t_{j-1}}}

\newcommand{\tkmj}{{t_{k-1}}}
\newcommand{\tkmd}{{t_{k-2}}}

\newcommand{\tkpj}{{t_{k+1}}}

\renewcommand{\dj}{{\Delta_j}}
\newcommand{\dk}{{\Delta_k}}
\newcommand{\djmj}{{\Delta_{j-1}}}
\newcommand{\djmd}{{\Delta_{j-2}}}

\newcommand{\dkmj}{{\Delta_{k-1}}}
\newcommand{\dkmd}{{\Delta_{k-2}}}

\newcommand{\intdj}{\int_{\dj}}
\newcommand{\intdk}{\int_{\dk}}
\newcommand{\intnn}{\int_0^\infty}
\newcommand{\intsn}{\int_s^\infty}
\newcommand{\inttn}{\int_t^\infty}
\newcommand{\intdjmj}{\int_{\Delta_{j-1}}}
\newcommand{\intdjmd}{\int_{\Delta_{j-2}}}
\newcommand{\intdjmt}{\int_{\Delta_{j-3}}}
\newcommand{\intdkmj}{\int_{\Delta_{k-1}}}
\newcommand{\intdkmd}{\int_{\Delta_{k-2}}}
\newcommand{\intdkmt}{\int_{\Delta_{k-3}}}
\newcommand{\intdkpj}{\int_{\Delta_{k+1}}}

\newcommand{\ro}{\varrho}

\newcommand{\f}{f^*}
\newcommand{\g}{g^*}
\newcommand{\h}{h^*}
\newcommand{\ff}{f^{**}}
\renewcommand{\gg}{g^{**}}

\newcommand{\lt}{\left(}
\newcommand{\rt}{\right)}

\newcounter{bcount}
\newcommand{\Bdef}[1]{\refstepcounter{bcount}\label{bcounter#1}}
\newcommand{\B}[1]{B_{\ref{bcounter#1}}}

\newtoks\by
\newtoks\paper
\newtoks\book
\newtoks\jour
\newtoks\yr
\newtoks\pages
\newtoks\vol
\newtoks\publ
\newtoks\eds
\newtoks\proc
\def\ota{{\hbox{???}}}
\def\cLear{\by=\ota\paper=\ota\book=\ota\jour=\ota\yr=\ota
\pages=\ota\vol=\ota\publ=\ota}
\def\endpaper{\textsc{\the\by}, \textit{\the\paper},
{\the\jour} \textbf{\the\vol} (\the\yr), \the\pages.\cLear}
\def\endbook{\textsc{\the\by}, {\the\book}, \the\publ.\cLear}
\def\endprep{\textsc{\the\by}, \textit{\the\paper}, \the\jour.\cLear}
\def\endproc{\textsc{\the\by}, \textit{\the\paper}, \the\publ, \the\pages.\cLear}
\def\name#1#2{#1\,#2}
\def\et{ and }

\begin{document}
\title{Embeddings and associated spaces of Copson-Lorentz spaces}
\author{Martin K\v{r}epela}
\address{Department of Applied Mathematics, University of Freiburg, Ernst-Zermelo-Sra\ss e 1, 79104 Freiburg im Breisgau, Germany}
\email{martin.krepela@math.uni-freiburg.de}
\begin{abstract}
  Let $m,p,q\in(0,\infty)$ and let $u,v,w$ be nonnegative weights. We characterize validity of the inequality 
    \[
      \lt \intnn w(t) (\f(t))^q \dt \rt^\jq \le C \lt \intnn v(t) \lt \inttn u(s) (\f(s))^m \ds \rt^\pm \!\! \dt \rt^\jp
    \]
  for all measurable functions $f$ defined on $\R^n$ and provide equivalent estimates of the optimal constant $C>0$ in terms of the weights and exponents. The obtained conditions characterize the embedding of the Copson-Lorentz space $\CL$, generated by the functional
    \[
      \|f\|_\CL := \lt \intnn v(t) \lt \inttn u(s) (\f(s))^m \ds \rt^\pm \!\! \dt \rt^\jp,
    \]
  into the Lorentz space $\Lqw$. Moreover, the results are applied to describe the associated space of the Copson-Lorentz space $\CL$ for the full range of exponents $m,p\in(0,\infty)$.
\end{abstract}   
\maketitle

\section{Introduction}

Let us start by presenting the notation used throughout this paper. An~introductory summary of the contents and purpose of the paper may be found further below in this section.\\

The cone of all real-valued measurable functions on $\R^n$ is denoted by $\M$. Next, $\MM$ denotes the cone of all locally integrable nonnegative functions on the interval $(0,\infty)$. \emph{A~weight} is any nonnegative measurable function defined on $(0,\infty)$.

The symbol $A\lesssim B$ means that there exists a~constant $\gamma\in(0,\infty)$ ``independent of relevant quantities in $A$ and $B$'' such that $A\le \gamma B$. More precisely, such constant $\gamma$ depends only on exponents $m,p,q$, unless specified else. It is written $A\approx B$ if both $A\lesssim B$ and $B\lesssim A$ hold. 

If $p\in(0,\infty)$ and $p\neq 1$, \emph{the conjugated exponent} $p'$ is defined by $p':=\frac p{p-1}$. If $0<p<1$, the exponent $p'$ is therefore negative. Conjugated exponents $m'$ and $q'$ are defined analogously.

Let $f\in\M$. Then $\f$ denotes \emph{the nonincreasing rearrangement of} $f$ defined by
  \[
    \f(t):=\inf \left\{ s\ge 0;\ \left|\left\{ x\in\R^n;\ |f(x)|>s \right\}\right|\le t \right\}
  \]
for all $t>0$. In here, the symbol $|E|$ stands for the ($n$-dimensional Lebesgue) measure of a~set $E\subset\R^n$. Next, \emph{the Hardy-Littlewood maximal function of $\f$} is defined by
  \[
    \ff(t) := \frac1t \int_0^t \f(s)\ds,
  \]
for all $t>0$. 

Let $m,p\in(0,\infty)$ and let $u,v$ be weights. Recall the definitions of \emph{the weighted Lorentz spaces} $\Lambda$ and $\Gamma$ which read as follows. 
  \begin{align*}  
    \Lambda^p(v)& := \left\{ f\in\M;\ \|f\|_{\Lambda^p(v)} := \lt \intnn v(t)(\f(t))^p \dt \rt^\jp <\infty \right\}, \\
    \Gamma^p(v)& := \left\{ f\in\M;\ \|f\|_{\Gamma^p(v)} := \lt \intnn v(t)(\ff(t))^p \dt \rt^\jp <\infty \right\}.
  \end{align*}
The notion of a~$\Gamma$ space may be further generalized, giving
  \[
    \Gamma^{m,p}(u,v) := \left\{ f\in\M;\ \|f\|_{\Gamma^{m,p}(u,v)} := \lt \intnn v(t) \lt \int_0^t u(s)(\f(s))^m \ds \rt^\frac pm \!\!\dt \rt^\jp \right\}.
  \]
The main object of interest in this paper is \emph{the Copson-Lorentz space} $\CL$, defined by
  \[
    \CL := \left\{ f\in\M;\ \|f\|_{\CL}:= \lt \intnn v(t) \lt \inttn u(s) (\f(s))^m \ds \rt^\pm \!\! \dt \rt^\jp <\infty \right\}.
  \]
This structure is labeled a~``space'' despite not necessarily being a~linear set, let alone a~normed linear space. An~example of such a~pathological case is constructed by letting $m=p=1$, $u(t)=1$, $v(t):=e^t$ and $V(t):=\int_0^t v(s)\ds = e^t-1$ for all $t>0$. Then $\CL=\Lambda^1(V)$ and 
  \[
    \lim_{t\to\infty} \frac{\int_0^{2t} V(s)\ds}{\int_0^t V(s)\ds} = \lim_{t\to\infty} \frac{e^{2t}-2t-1}{e^t-t-1} = \infty.
  \]
By the characterization of linearity of $\Lambda$ spaces \cite[Theorem 1.4]{CKMP}, $\CL$ is not a~linear set.

The $\Lambda$ spaces are not always linear sets either, as seen above. The $\Gamma$ spaces are at least linear spaces thanks to sublinearity of the mapping $f\to\ff$, the functional $\|\cdot\|_{\Gamma^p(v)}$ however does not have to be a~norm (see \cite{KM,Sou}). Nevertheless, the term ``space'' is used to describe all these structures, for the sake of simplicity.

Suppose that $\|\cdot\|_X : \M\to[0,\infty]$ is a~functional such that $\|\lambda f \|_X = |\lambda|\|f\|_X$ and $\|f\|_X\le\|g\|_X$ for all $\lambda\in[0,\infty)$ and $f,g\in\M$ such that $|f|\le |g|$ a.e.~on $\R^n$. Let $\|\cdot\|_Y$ be another functional with the same properties. Let $X$, $Y$ be two function ``spaces'' given by $X=\{f\in\M;\ \|f\|_X<\infty\}$ and $Y=\{f\in\M;\ \|f\|_Y<\infty\}$. Then $X$ is said to be \emph{embedded in} $Y$ and it is written $X\hra Y$ if there exists a~constant $C\in(0,\infty)$ such that
  \begin{equation}\label{115}
    \|f\|_Y \le C \|f\|_X
  \end{equation}
for all $f\in\M$. The infimum of all $C\in(0,\infty]$ such that \eqref{115} holds for all $f\in\M$ is called \emph{the optimal constant}. Hence, if $X$ is not embedded in $Y$, the optimal constant in \eqref{115} is infinite. 

Assume that $X$ is moreover \emph{rearrangement-invariant}, i.e., that $\|f\|_X=\|h\|_X$ whenever $f,h\in\M$ are such that $\f=\h$ on $(0,\infty)$. Then \emph{the associated space of} $X$, denoted $X'$, is defined as $X'=\{g\in\M;\ \|g\|_{X'}<\infty\}$, where
  \[  
    \|g\|_{X'}:= \sup \left\{ \intnn \f(t)\g(t)\dt;\ \|f\|_X\le 1 \right\}.
  \]
If $X$ is a~Banach function space (see \cite{BS}), then $X'$ is a~Banach function space as well. However, to define $X'$ in the way described above is possible even for a~more general $X$, though without claiming space-like properties of $X'$. For more details, see \cite{BS}.

As it can be observed from the previous definition, embeddings into $\Lambda$ spaces play a~rather significant role since the associate ``norm'' $\|g\|_X$ of a~function $g\in\M$ is equal to the optimal constant related to the embedding $X\hra \Lambda^1(\g)$. Hence, one gets a~description of the associate space to $X$ once the embedding $X\hra \Lambda^1(w)$ has been reasonably characterized.

Unlike the converse embedding $\Lambda\hra X$, which in case of rearrangement-invariant $X$ can be approached by the so-called reduction theorems (see \cite{GS}), the embedding $X\hra\Lambda$ is comparatively hard to characterize. A~great difficulty arises especially when the functional $\|\cdot\|_X$ involves a~more complicated operator. Results covering such problems in a~reasonable way have been rather sparse so far.

One of the few examples of rearrangement-invariant spaces whose embedding into $\Lambda$ was successfully and satisfactorily characterized are the $\Gamma$ spaces, including their generalized variants. The results of this type were obtained in \cite{GP,GPS} by a~method of discretization. Various steps in this direction were made even earlier, see e.g.~\cite{GHS,CPSS}. However, the older results included implicit and/or discrete conditions which were of little use. The main achievement of \cite{GS} was therefore the so-called ``anti-discretization'' by which one could turn the discrete expressions to explicit integral conditions.\\

The purpose of this paper is to prove a~full characterization of the embedding $\CL\hra\Lambda^q(w)$ for $m,p,q\in(0,\infty)$. In other words, the goal is to give necessary and sufficient conditions on the weights $u,v,w$ and exponents $m,p,q$ such that the inequality 
  \begin{equation}\label{30}
    \lt \intnn w(t) (\f(t))^q \dt \rt^\jq \le C \lt \intnn v(t) \lt \inttn u(s) (\f(s))^m \ds \rt^\pm \!\! \dt \rt^\jp
  \end{equation}
holds for all functions $f\in\M$. Moreover, equivalent estimates of the optimal constant $C$ in this inequality are provided. Very little assumptions on the weights $u,v$ are taken here and these, in fact, only assure that the space $\CL$ is not ``degenerate'' (see Section 3). 

Notice also that for $p=m$ the right-hand side of \eqref{30} becomes the ``norm'' of $f$ in $\Lambda^p(\varrho)$ with the weight $\varrho$ defined as $\varrho(t):=u(t)\int_0^t v(s)\ds$ for $t>0$. In this case, inequality \eqref{30} corresponds to a~``$\Lambda\hra\Lambda$'' embedding, and characterizations of its validity are therefore already known (see \cite{AM,Sa,CPSS} and the references therein).

The proofs in this paper are based on a~discretization technique which is developed to fit the Copson-Lorentz functional. Hence, it differs from the one used in \cite{GP}, although the methods of this paper and \cite{GP} share certain root ideas and have some common features. The conditions obtained here have an~explicit integral form, so the ``anti-discretization'' process makes a~significant part of the work done here. 

To the author's knowledge, not many attempts have been done on solving the problems ``$X\hra\Lambda$'' by discretizing the right-hand side of the requested inequalities such as \eqref{30} in cases when the right-hand side is not a~$\Gamma$ norm or its generalization. This paper shows that it is possible. It suggests that, with a~proper modification, the discretization is an~effective method which can be rather universally used to handle inequalities involving rearrangement-invariant functionals. It is also shown here how such a~modification can be made in order to fit a~particular functional (in this case the Copson-Lorentz one).

Since the estimates of the optimal constant $C$ in \eqref{30} are obtained, they are also applied here directly to provide a~characterization of the associate ``norm'' $\|\cdot\|_{(\CL)'}$. Results of this type are very desirable. For instance, the ``dual'' expression of a~norm can be used to deal with inequalities involving very complicated operators. A~recent example of this is the use of a~duality method in \cite{GKPS} to solve a~notoriously difficult problem of comparison of generalized $\Gamma$ spaces.

As for the structure for the paper, Section 2 consists of a~list of known basic results which are used throughout the article. Section 3 is devoted to proving the core discretization results, in particular the equivalent expression of the Copson-Lorentz functional $\|\cdot\|_{\CL}$ by means of a~discretizing sequence, i.e., a~special partition of $(0,\infty)$, and the fundamental function of $\CL$. The main results are obtained in Section 4 where all the characterizing conditions of the embedding $C\!L\hra\Lambda$ can be found. Finally, the description of $(\CL)'$, i.e., the associate space of $\CL$, is given in Section 5.

\section{Preliminaries}

Some details and proofs of the following auxillary propositions may be found, for example, in \cite{GHS,GP,K}.

\begin{prop}[H\"older inequality]\label{40}
  Let $k_{\min},k_{\max}\in\Z\cup\{\plusminus\infty\}$ be such that $k_{\min}<k_{\max}$. Let $\{a_k\}_{k=k_{\min}}^{k_{\max}}$ and $\{b_k\}_{k=k_{\min}}^{k_{\max}}$ be two nonnegative sequences. Assume that $0<q<p<\infty$. Then
    \[
      \left(\sum_{k=k_{\min}}^{k_{\max}} a_k^q b_k\right)^\frac1q \le \left(\sum_{k=k_{\min}}^{k_{\max}} a_k^p\right)^\frac1p \left( \sum_{k=k_{\min}}^{k_{\max}} b_k^\frac{p}{p-q}\right)^\frac{p-q}{pq}.
    \]  
  Moreover, there exists a~nonnegative sequence $\{c_k\}_{k=k_{\min}}^{k_{\max}}$ such that $\sum_{k=k_{\min}}^{k_{\max}} c_k^p=1$ and
    \[
      \left( \sum_{k=k_{\min}}^{k_{\max}} b_k^\frac{p}{p-q}\right)^\frac{p-q}{pq} = \left(\sum_{k=k_{\min}}^{k_{\max}} c_k^q b_k\right)^\frac1q.
    \]
\end{prop}

\begin{prop}\label{20}
  Let $0<\alpha<\infty$ and $1<D<\infty$. Then there exists a~constant $C_{\alpha,D}\in(0,\infty)$ such that for any $k_{\min},k_{\max}\in\Z\cup\{\plusminus\infty\}$, $k_{\min}<k_{\max}$, and any two nonnegative sequences $\{b_k\}_{k=k_{\min}}^{k_{\max}}$ and $\{c_k\}_{k=k_{\min}}^{k_{\max}}$, satisfying $b_{k+1}\ge D\, b_k$  for all $k\in\Z,\ k_{\min}\le k <k_{\max}$, there holds
    \[
      \sum_{k=k_{\min}}^{k_{\max}} \left( \sum_{m=k}^{k_{\max}} c_m \right)^\alpha b_k  \le C_{\alpha,D} \sum_{k=k_{\min}}^{k_{\max}} c_k^\alpha b_k, 
    \]
    \[
      \sum_{k=k_{\min}}^{k_{\max}} \left(\sup_{k\le m \le k_{\max}} \!\! c_m \right)^\alpha b_k  \le C_{\alpha,D}  \sum_{k=k_{\min}}^{k_{\max}} c_k^\alpha b_k 
    \]
  and
    \[
      \sup_{k_{\min}\le k \le k_{\max}} \left( \sum_{m=k}^{k_{\max}} c_m \right)^\alpha b_k  \le C_{\alpha,D} \sup_{k_{\min}\le k \le k_{\max}} c_k^\alpha b_k. 
    \]
\end{prop}    

\begin{prop}\label{21}
  Let $0<\alpha<\infty$ and $1<D<\infty$. Then there exists a~constant $C_{\alpha,D}\in(0,\infty)$ such that for any $k_{\min},k_{\max}\in\Z\cup\{\plusminus\infty\}$, $k_{\min}<k_{\max}$, and any two nonnegative sequences $\{b_k\}_{k=k_{\min}}^{k_{\max}}$ and $\{c_k\}_{k=k_{\min}}^{k_{\max}}$, satisfying $b_k\ge D\,b_{k+1}$ for all $k\in\Z,\ k_{\min}\le k <k_{\max}$, there holds
    \[
      \sum_{k=k_{\min}}^{k_{\max}} \left( \sum_{m=k_{\min}}^{k} c_m \right)^\alpha b_k  \le C_{\alpha,D} \sum_{k=k_{\min}}^{k_{\max}} c_k^\alpha b_k, 
    \]
    \[
      \sum_{k=k_{\min}}^{k_{\max}} \left(\sup_{k_{\min}\le m \le k} \!\! c_m \right)^\alpha b_k  \le C_{\alpha,D}  \sum_{k=k_{\min}}^{k_{\max}} c_k^\alpha b_k 
    \]
  and
    \[
      \sup_{k_{\min}\le k \le k_{\max}} \left( \sum_{m=k_{\min}}^{k} c_m \right)^\alpha b_k  \le C_{\alpha,D} \sup_{k_{\min}\le k \le k_{\max}} c_k^\alpha b_k. 
    \]
\end{prop}  

Another basic ingredient is the case of the Hardy inequality shown below. For its proof see \cite{SS} and the references therein.

\begin{prop}[Hardy inequality]\label{35}
  Let $a,b\in[0,\infty]$ and let $\eta,\ro$ be weights. 
  
  {\rm(i)}
        Let $1\le q<\infty$. Then the inequality
          \[
            \lt \int_a^b \lt \int_t^b h(s)\ds \rt^q \ro(t)\dt \rt^\jq \lesssim \int_a^b h(t)\eta(t) \dt\, \esssup_{t\in(a,b)} \lt \int_a^t \ro(s)\ds \rt^\jq \eta^{-1}(t)  
          \]
        holds for all $h\in\MM(a,b)$. Moreover, there exists a~function $g\in\MM(a,b)$ such that $\int_a^b g\eta = 1$ and 
          \[
            \esssup_{t\in(a,b)} \lt \int_a^t \ro(s)\ds \rt^\jq \eta^{-1}(t) \lesssim  \lt \int_a^b \lt \int_t^b g(s)\ds \rt^q \ro(t)\dt \rt^\jq.
          \]
  
  {\rm(ii)}
        Let $0<q<1$. Then the inequality
          \[
            \lt \int_a^b \lt \int_t^b h(s)\ds \rt^q \ro(t)\dt \rt^\jq \lesssim \int_a^b h(t)\eta(t) \dt\, \lt \int_a^b \lt \int_a^t \ro(s) \ds \rt^{-q'} \!\!\!\! \ro(t) \esssup_{x\in(t,b)} \eta^{q'}(x)\,dt \rt^{-\frac1{q'}}  
          \]
        holds for all $h\in\MM(a,b)$. Moreover, there exists a~function $g\in\MM(a,b)$ such that $\int_a^b g\eta = 1$ and 
          \[
            \lt \int_a^b \lt \int_a^t \ro(s) \ds \rt^{-q'} \!\!\!\!\ro(t) \esssup_{x\in(t,b)} \eta^{q'}(x)\,dt \rt^{-\frac1{q'}} \lesssim  \lt \int_a^b \lt \int_t^b g(s)\ds \rt^q \ro(t)\dt \rt^\jq.
          \]
\end{prop}

\section{Discretization of the Copson-Lorentz functional}

What follows is the core of the discretization method used in this article. The results of this section may be also used to deal with other questions related to the Copson-Lorentz spaces and similar problems.

\begin{df}\label{13}
  Let $u,v$ be weights and $m,p\in(0,\infty)$. Let $\fii:[0,\infty]\to[0,\infty]$ be a~mapping defined by
    \[
      \fii(t):= \lt \int_0^t v(s)\lt \int_s^t u(x)\dx \rt^\pm \ds \rt^\jp
    \]
  for every $t\in(0,\infty]$. Then $\fii$ is called \emph{the fundamental function of} $\CL$.
\end{df}

One can easily observe that the value of the fundamental function of $\CL$ at a~point $t>0$ is equal to $\|\chi_E\|_{\CL}$, where $E$ is any subset of $\R^n$ of measure $t$ and $\chi_E$ is its characteristic function. The definition above therefore corresponds to the standard terminology used in the context of rearrangement-invariant spaces (cf.~\cite{BS}).

From now on, certain assumptions will be imposed on the weights $u,v$. As it can be seen below, these assumptions are reasonable and exclude only some ``degenerate'' cases of $\CL$.

\begin{df}\label{8}
  Let $m,p\in(0,\infty)$ and let $u,v$ be weights. Then $(u,v)$ is called \emph{an~admissible pair of weights with respect to} $(m,p)$ if $0<\fii(t)<\infty$ for all $t\in(0,\infty)$.
\end{df}

\begin{rem}\label{6}
  (i) Let $m,p\in(0,\infty)$, let $(u,v)$ be an~admissible pair of weights with respect to $(m,p)$ and let $\fii$ be the fundamental function of $\CL$. Then $\fii$ is continuous and nondecreasing on $(0,\infty)$ and $\lim_{t\to 0+}\fii(t)=0$. The last fact may be verified by writing
    \begin{align*}
      \infty > \fii^p(t) & = \int_0^t v(s) \lt\int_s^t u(x)\dx\rt^{\pm} \hspace{-5pt} \ds = \pm \int_0^t v(s) \int_s^t \lt \int_s^x u(y)\dy \rt^{\pm-1} \hspace{-10pt} u(x)\dx \ds \\
              & = \pm \int_0^t u(x) \int_0^x v(s) \lt \int_s^x u(y)\dy \rt^{\pm-1} \hspace{-10pt} \ds \dx. 
    \end{align*}
  Since the integrand in the last expression does not depend on $t$, the value of the integral converges to zero as $t\to 0+$. Besides that, the derivative $\fii'$ exists a.e.~on $(0,\infty)$ and the identity
    \begin{equation}\label{44}
      \fii'(t) = \frac{p^{\jp-1}}{m^\jp}\, \fii^{1-p}(t)\, u(t) \int_0^t v(s) \lt \int_s^t u(y)\dy \rt^{\pm-1} \hspace{-10pt} \ds 
    \end{equation}
  holds for a.e.~$t\in(0,\infty)$ with the convention
    \begin{equation}\label{116}
      \frac1\infty := 0, \qquad 0.\infty := 0, \qquad 0^0 := 1
    \end{equation}
  in force to prevent undefined terms from appearing in case that $p\le m$ and $u$ is equal to zero on a~nontrivial interval. The function $\fii$ also has the mean value property. In particular, for every $t\in(0,\infty)$ there exists an~$x\in(0,t)$ such that $\fii^p(x)=2^{-\pm-1}\fii^p(t)$ and if, moreover, $\fii(\infty)=\infty$, then there also exists an~$s\in(t,\infty)$ such that $\fii^p(s)=2^{\pm+1}\fii^p(t)$. Obviously, the function $\int_0^\cdot v(s)\ds$ has the mean value property as well.

  (ii) If $m,p\in(0,\infty)$ and a~pair of weights $(u,v)$ is not admissible with respect to $(m,p)$, then there exists a~$t\in(0,\infty)$ such that either $\fii(t)=0$ or $\fii(t)=\infty$. If so, there exists a~set $E\in\R^n$ of finite positive measure and such that $\|\chi_E\|_{\CL}=0$ or $\|\chi_E\|_{\CL}=\infty$. In either case, it excludes the possibility of $\CL$ being a~Banach function space in the sense of Luxemburg's definition (see \cite{BS} for details). 
\end{rem} 

\begin{thm}\label{1}
  Let $m,p\in(0,\infty)$, let $(u,v)$ be an~admissible pair of weights with respect to $(m,p)$ and let $\fii$ be the fundamental function of $\CL$. Then there exist a~$K\in\{0,\infty\}$, a~set 
	  \begin{equation}\label{defk}
		  \K:=\begin{cases}
			  \{k\in\Z,~k\le0\} & \text{if }K=0,\\
			  \Z & \text{if }K=\infty,
			\end{cases} 
	  \end{equation}
	sets $\Kj$ and $\Kd$ such that $\Kj\cap\Kd=\varnothing$, $\Kj\cup\Kd=\K$ and a~sequence $\{\tk\}_{k\in\K}$ with the following properties: $0<t_{k-1}\le\tk<\infty$ is satisfied for all $k\in\K\setminus\{K\}$; the term $t_0$ is defined as $\infty$ in case that $K=0$; the inequalities
    \begin{equation} \label{2}
      \int_0^{t_{k}} v(t)\dt \ge 2^{\pm+1} \int_0^{t_{k-1}} v(t)\dt,
    \end{equation}
    \begin{equation} \label{3}
      \fii^p(\tk) \ge 2^{\pm+1} \fii^p(t_{k-1}) \dt 
    \end{equation}
  hold for all $k\in\K$; the identity 
    \begin{equation} \label{4}
      \int_0^{t_{k}} v(t)\dt = 2^{\pm+1} \int_0^{t_{k-1}} v(t)\dt
    \end{equation}
  holds for all $k\in\Kj$; the identity 
    \begin{equation} \label{5}
      \fii^p(\tk) = 2^{\pm+1} \fii^p(t_{k-1})
    \end{equation}
  holds for all $k\in\Kd$.
\end{thm}

\begin{proof}
  \emph{Step 1.} 
    If $\intnn v(t)\dt<\infty$ or $\fii(\infty)<\infty$, put $K:=0$ and $t_0:=\infty$. Otherwise, put $K:=\infty$ and $t_0:=1$. Define the set $\K$ by \eqref{defk}. 
  
  \emph{Step 2.}
    Suppose that $k\in\Z,$ $k\le0$ and $t_k$ is defined. There exist (cf.~Remark \ref{6}(i)) points $x_{k},y_{k}\in(0,\tk)$ such that
      \[
        \int_0^{x_{k}}v(t)\dt = \frac1{2^{\pm+1}}\int_0^{t_k}v(t)\dt \quad \textnormal{and}\quad \fii^p(y_{k})= \frac1{2^{\pm+1}}\fii^p(t_k).
      \]
    Define $t_{k-1}:=\min\{x_k,y_k\}$. Then \eqref{2} and \eqref{3} are both satisfied and one of the identities \eqref{4} and \eqref{5} holds true as well. One continues by induction, replacing $k$ with $k-1$ and repeating Step 2. In this manner, the part of the sequence $\{t_k\}$ indexed by nonpositive integers is constructed.
    
  \emph{Step 3.}
    Suppose that $k\in\N$, $t_{k-1}$ is defined and $t_{k-1}<\infty$. There exist (cf.~Remark \ref{6}(i)) points $z_k,s_k\in(t_{k-1},\infty)$ such that 
      \[
        \int_0^{z_k}v(t)\dt = 2^{\pm+1} \int_0^{t_{k-1}}v(t)\dt \quad\textnormal{and}\quad \fii^p(s_k)=2^{\pm+1}\fii^p(t_{k-1}).
      \]
    Define $t_k:=\max\{z_k,s_k\}$. Then the conditions \eqref{2}, \eqref{3} and one of \eqref{4} and \eqref{5} are satisfied. Again, one proceeds by induction, replacing $k$ with $k+1$ and repeating Step 3. Notice that if $K=0$, Step 3 is never performed. If $K=\infty$, the part of the sequence $\{t_k\}$ indexed by positive integers is constructed.
  
  \emph{Step 4.}
    Define 
      \[
        \Kj:=\{k\in\Z,\ k\le0,\ t_{k-1}=x_k\} \cup \{k\in\N,\ t_k=z_k\}
      \]
    and $\Kd:=\K\setminus \Kj$. Obviously, $\Kj\cap\Kd=\varnothing$, $\K=\Kj\cup\Kd$, and, thanks to the construction of $\{t_k\}_{k\in\K}$, identity \eqref{4} holds for all $k\in\Kj$ while \eqref{5} holds for all $k\in\Kd.$
\end{proof} 

\begin{df}\label{9}
  In the setting of Theorem \ref{1}, the sequence $\{\tk\}_{k\in\K}$ obtained there is called \emph{a~discretizing sequence of} $\CL$. For each $k\in\K$ denote 
    \[
      \dk:=[\tk,t_{k+1}]
    \]
  and
    \[
      U(\dk):=U(\tk,t_{k+1}),
    \]
  where
    \[  
      U(s,t):=\int_s^t u(x)\dx,
    \]
  whenever $0\le s\le t\le \infty$. Furthermore, define
    \[
      \K-1 :=\{k\in\Z,\ k+1\in\K\}, \quad \K-2 :=\{k\in\Z,\ k+2\in\K\}.
    \] 
\end{df}

\begin{rem}
  In the setting of Theorem \ref{1}, admissibility of $(u,v)$ with respect to $(m,p)$ ensures that $\lim_{k\to-\infty} t_k = 0$ and $\lim_{k\to K} t_k = \infty$ (in case of $K=0$, $\lim_{k\to 0} t_k$ is defined as $t_0$). Hence, the identity 
    \[
      \intnn \sigma(t)\dt = \sumk \intdkmj \sigma(t)\dt
    \]
  holds for any nonnegative measurable function $\sigma$ defined on $(0,\infty)$.
\end{rem}

\begin{thm}\label{10}
  Let $m,p\in(0,\infty)$, let $(u,v)$ be an~admissible pair of weights with respect to $(m,p)$, let $\fii$ be the fundamental function of $\CL$ and assume that $\{\tk\}_{k\in\K}$ is a~discretizing sequence of $\CL$. Then, for every $k\in\K$ and $t\in\dkmj$, the estimates
    \begin{equation}\label{11}
      \int_0^{\tk} v(s)\ds \le \frac{2^{\pm+1}}{2^{\pm+1}-1} \intdkmj v(s)\ds
    \end{equation}
  and
    \begin{align}
      \fii^p(t) & \le \frac{2^{\frac{3p}m+3}}{2^{\pm+1}-1} \intdkmt v(s)\ds\ \Upm(\dkmd)\  +\ 2^{\pm+2} \intdkmd v(s)\Upm(s,\tkmj)\ds \label{12}\\
              & \quad +  \frac{3\cdot 2^{\frac{2p}m+1}}{2^{\pm+1}-1} \int_\dkmd v(s)\ds\ \Upm(\tkmj,t)  \nonumber
    \end{align}    
  hold true.
\end{thm}  

\begin{proof}
  Let $k\in\K$. By \eqref{2}, one gets
    \[  
      \intdkmj v(s)\ds = \int_0^\tk v(s)\ds - \int_0^{\tkmj} v(s) \ds \ge \lt 1-\frac1{2^{\pm+1}}\rt \int_0^\tk v(s)\ds,
    \]
  which proves \eqref{11}. 
  
  Now let $t\in\dkmj$. If $k\in\Kd$, then \eqref{3} implies
    \[
      \fii^p(t) \le \int_0^{t_{k}} v(s)\Upm(s,t_{k})\ds = 2^{\pm+1} \int_0^\tkmj v(s)\Upm(s,\tkmj)\ds.
    \]
  If $k\in\Kj$, one gets the following estimates:
    \begin{align}
      \fii^p(t) & = \int_0^\tkmj v(s)\Upm(s,t)\ds + \int_{\tkmj}^t v(s)\Upm(s,t)\ds \nonumber \\
              & \le 2^\pm \int_0^\tkmj v(s)\Upm(s,\tkmj)\ds + 2^\pm \int_0^\tkmj v(s)\ds\ \Upm(\tkmj,t) + \int_0^{\tk} v(s)\ds\ \Upm(\tkmj,t) \label{15}\\
              & = 2^\pm \int_0^\tkmj v(s)\Upm(s,\tkmj)\ds + 3\cdot 2^\pm \int_0^\tkmj v(s)\ds\ \Upm(\tkmj,t) \label{16} \\
              & \le 2^\pm \int_0^\tkmj v(s)\Upm(s,\tkmj)\ds + \frac{3\cdot 2^{\frac{2p}m+1}}{2^{\pm+1}-1} \int_\dkmd v(s)\ds\ \Upm(\tkmj,t). \label{17} 
    \end{align}
  In here, estimate \eqref{15} follows from the inequality $(a+b)^\pm\le 2^{\pm}\lt a^\pm+b^\pm\rt$ valid for $a,b\ge0$, identity \eqref{16} follows from \eqref{4}, and estimate \eqref{17} follows from the inequality \eqref{11} (in which $k$ is replaced by $k-1$). Combining the obtained estimates, one gets
    \begin{equation}\label{14}
      \fii^p(t) \le 2^{\pm+1} \int_0^\tkmj v(s)\Upm(s,\tkmj)\ds + \frac{3\cdot 2^{\frac{2p}m+1}}{2^{\pm+1}-1} \int_\dkmd v(s)\ds\ \Upm(\tkmj,t).
    \end{equation}
  Next, one has
    \begin{align*}
      & \int_0^\tkmj v(s)\Upm(s,\tkmj)\ds \\
        & \quad \le 2^\pm \int_0^\tkmd v(s)\Upm(s,\tkmd)\ds + 2^\pm \int_0^\tkmd v(s)\ds\ \Upm(\dkmd) + \intdkmd v(s)\Upm(s,\tkmj)\ds \nonumber\\
        & \quad \le \frac12 \int_0^\tkmj v(s)\Upm(s,\tkmj)\ds + \frac{2^{\frac{2p}m+1}}{2^{\pm+1}-1} \intdkmt v(s)\ds\ \Upm(\dkmd) + \intdkmd v(s)\Upm(s,\tkmj)\ds,
    \end{align*}
  where the second inequality follows from \eqref{3} and \eqref{11} (with $k$ replaced by $k-1$ and $k-2$, respectively). Therefore,
    \[
      \int_0^\tkmj v(s)\Upm(s,\tkmj)\ds \le  \frac{2^{\frac{2p}m+2}}{2^{\pm+1}-1} \intdkmt v(s)\ds\ \Upm(\dkmd) + 2 \intdkmd v(s)\Upm(s,\tkmj)\ds.
    \]
  Together with \eqref{14}, this gives the result \eqref{12}.
\end{proof}

The following theorem is the main result concerning the discretization of the Copson-Lorentz functional $\|\cdot\|_{\CL}$. It is stated in a~general form for $m\in(0,\infty)$ although the sole equivalence \eqref{19} with $m=1$ is sufficient for the purposes of this paper. 

\begin{thm}\label{18}
  Let $m,p\in(0,\infty)$, let $(u,v)$ be an~admissible pair of weights with respect to $(m,p)$, let $\fii$ be the fundamental function of $\CL$ and $\{\tk\}_{k\in\K}$ be a~discretizing sequence of $\CL$. Then 
    \begin{align}
      & \intnn v(t) \lt \inttn u(s) \lt \intsn h(y)\dy \rt^m \!\! \ds \rt^\pm \!\! \dt \approx \sumk \lt\ \intdkmj \fii^m(y) \lt \int_y^\tk h(x)\dx \rt^{m-1} \hspace{-6pt} h(y)\dy \rt^\pm  \label{19} \\
      & \quad \approx \sumk \ \fii^p(\tkmj) \lt\ \intdkmj \! h(y)\dy \rt^p + \sumk \lt\ \intdkmj \fii^{m-1}(y) \fii'(y) \lt \int_y^\tk h(x)\dx \rt^m \hspace{-4pt} \dy \rt^\pm \label{117}
    \end{align}
  for all $h\in\MM$, with the convention \eqref{116} applied whenever needed.
\end{thm}

\begin{proof}
  Let $h\in\MM$. Then
    \Bdef{1}\Bdef{2}\Bdef{3}
    \begin{align*}
      & \intnn v(t) \lt \inttn u(s) \lt \intsn h(y)\dy \rt^m \!\! \ds \rt^\pm \!\! \dt \\
      & = \sumk\, \intdkmj v(t) \lt \inttn u(s) \lt \intsn h(y)\dy \rt^m \!\! \ds \rt^\pm \!\! \dt \\
      & \approx \sumk\, \intdkmj v(t) \lt \int_t^\tk u(s) \lt \int_s^\infty h(y)\dy \rt^m \!\! \ds \rt^\pm \!\! \dt + \sumkmj\, \intdkmj v(t)\dt \lt \int_\tk^\infty u(s)\lt \int_s^\infty h(y)\dy \rt^m\!\!\ds\rt^\pm \\
      & \approx \sumk\, \intdkmj v(t) \lt \int_t^\tk u(s) \lt \int_s^\tk h(y)\dy \rt^m \!\! \ds \rt^\pm \!\! \dt + \sumkmj\, \intdkmj v(t) \Upm(t,\tk) \dt \lt \int_\tk^\infty h(y)\dy \rt^p \\ 
      & \quad + \sumkmj\, \intdkmj v(t)\dt \lt \int_\tk^\infty u(s)\lt \int_s^\infty h(y)\dy \rt^m\!\!\ds\rt^\pm \\
      & =: \B{1}+\B{2}+\B{3}.
    \end{align*}
  Since the function $h$ is locally integrable and the condition \eqref{116} is in force, the relation 
    \[
      \lt \int_a^b h(y)\dy \rt^m =\ m \int_a^b \lt \int_y^b h(x)\dx \rt^{m-1} \!\!\! h(y) \dy
    \]
  may be used for any $0\le a\le b\le\infty$. Thus, we may write
    \begin{align}
      \B{1} & \approx  \sumk\, \intdkmj v(t) \lt \int_t^\tk u(s) \int_s^\tk \lt \int_y^\tk h(x)\dx \rt^{m-1} \!\!\! h(y) \dy \ds \rt^\pm \!\! \dt \nonumber\\
            & =        \sumkj\, \intdkmj v(t) \lt \int_t^\tk u(s) \int_s^\tk \lt \int_y^\tk h(x)\dx \rt^{m-1} \!\!\! h(y) \dy \ds \rt^\pm \!\! \dt \nonumber\\
            & \quad +  \sumkd\, \intdkmj v(t) \lt \int_t^\tk u(s) \int_s^\tk \lt \int_y^\tk h(x)\dx \rt^{m-1} \!\!\! h(y) \dy \ds \rt^\pm \!\! \dt \nonumber\\
            & \le      \sumkj \int_0^\tk v(t)\dt  \lt\ \intdkmj u(s) \int_s^\tk \lt \int_y^\tk h(x)\dx \rt^{m-1} \!\!\! h(y) \dy \ds \rt^\pm  \nonumber\\
            & \quad +  \sumkd\, \fii^p(\tk) \lt\ \intdkmj \lt \int_y^\tk h(x)\dx \rt^{m-1} \!\!\! h(y) \dy  \rt^\pm \nonumber\\
            & \lesssim \sumkj \int_0^\tkmj v(t)\dt  \lt\ \intdkmj U(\tkmj,y) \lt \int_y^\tk h(x)\dx \rt^{m-1} \!\!\! h(y) \dy \rt^\pm \label{22}\\
            & \quad +  \sumkd\, \fii^p(\tkmj) \lt\ \intdkmj \lt \int_y^\tk h(x)\dx \rt^{m-1} \!\!\! h(y) \dy \rt^\pm  \nonumber\\
            & \lesssim \sumk\, \lt\ \intdkmj \fii^m(y) \lt \int_y^\tk h(x)\dx \rt^{m-1} \!\!\! h(y) \dy \rt^\pm. \nonumber
    \end{align}
  To get \eqref{22}, properties \eqref{4}, \eqref{5} and the Fubini theorem were used. Next, the term $\B{2}$ is estimated as follows.
    \begin{align}
      \B{2} & \le      \sumkmj\, \fii^p(\tk) \lt \int_\tk^\infty h(y)\dy \rt^p \nonumber\\
            & \lesssim \sumkmj\, \fii^p(\tk) \lt\ \intdk h(y)\dy \rt^p\label{23}\\
            & \approx  \sumkmj \fii^p(\tk) \lt\ \intdk \lt \int_y^\tkpj h(x)\dx \rt^{m-1} \!\!\! h(y) \dy \rt^\pm  \nonumber\\
            & \le      \sumkmj \lt\ \intdk \fii^m(y) \lt \int_y^\tkpj h(x)\dx \rt^{m-1} \!\!\! h(y) \dy \rt^\pm \nonumber \\
            & =        \sumk\, \lt\ \intdkmj \fii^m(y) \lt \int_y^\tk h(x)\dx \rt^{m-1} \!\!\! h(y) \dy \rt^\pm. \nonumber
    \end{align}
  In here, inequality \eqref{23} follows from Proposition \ref{20} and \eqref{3}. Concerning $\B{3}$, one has
    \begin{align}
      \B{3} & \lesssim \sumkmj\, \intdkmj v(t)\dt \lt\ \intdk u(s)\lt \int_s^\infty h(y)\dy \rt^m\!\!\ds\rt^\pm \label{24}\\
            & \lesssim \sumkmj\, \intdkmj v(t)\dt \lt\ \intdk u(s)\lt \int_s^\tkpj h(y)\dy \rt^m\!\!\ds\rt^\pm \nonumber\\
            & \quad +  \sumkmd\, \intdkmj v(t)\dt\ \Upm(\dk) \lt\ \int_{\tkpj}^\infty h(y)\dy \rt^p \nonumber\\
            & \lesssim \sumkmj\, \intdkmj v(t)\dt \lt \intdk u(s)\lt \int_s^\tkpj h(y)\dy \rt^m\!\!\ds\rt^\pm + \sumkmd\, \fii^p(\tkpj) \lt\, \int_{\tkpj}^\infty h(y)\dy \rt^p \nonumber\\
            & \lesssim \sumkmj\, \intdkmj v(t)\dt \lt\ \intdk u(s) \int_s^\tkpj \lt \int_y^\tkpj h(x)\dx \rt^{m-1}\!\!\! h(y)\dy \ds\rt^\pm \label{25}\\
            & \quad +  \sumkmd\, \fii^p(\tkpj) \lt\ \intdkpj h(y)\dy \rt^p \nonumber\\
            & \lesssim \sumkmj\, \intdkmj v(t)\dt \lt\ \intdk U(\tk,y) \lt \int_y^\tkpj h(x)\dx \rt^{m-1}\!\!\! h(y)\dy \rt^\pm \label{26}\\
            & \quad +  \sumkmd\, \fii^p(\tkpj) \lt\ \intdkpj \lt \int_y^{t_{k+2}} h(x)\dx\rt^{m-1} \!\!\! h(y)\dy \rt^\pm \nonumber\\
            & \lesssim \sumkmj \lt \intdk \fii^m(y) \lt \int_y^\tkpj h(x)\dx \rt^{m-1}\!\!\! h(y)\dy \rt^\pm \nonumber\\
            & \quad  + \sumkmd\, \lt\, \intdkpj \fii^m(y) \lt \int_y^{t_{k+2}} h(x)\dx\rt^{m-1} \!\!\! h(y)\dy \rt^\pm \nonumber\\
            & \approx  \sumk\, \lt\, \intdkmj \fii^m(y) \lt \int_y^\tk h(x)\dx \rt^{m-1} \!\!\! h(y) \dy \rt^\pm. \nonumber
    \end{align}
  Inequality \eqref{24} follows from Proposition \ref{20} and \eqref{2}, inequality \eqref{25} from Proposition \ref{20} and \eqref{3}. The Fubini theorem implies \eqref{26}. 
  
  The obtained estimates of $\B{1}$, $\B{2}$ and $\B{3}$ together yield the ``$\lesssim$'' inequality in \eqref{19}. Next step is to prove the converse inequality. It is done in the following way.
    \begin{align}
      & \sumk \lt\ \intdkmj \fii^m(y) \lt \int_y^\tk h(x)\dx \rt^{m-1} \hspace{-6pt} h(y)\dy \rt^\pm \nonumber\\
      & \lesssim \sumk \ \intdkmt v(t)\dt\ \Upm(\dkmd) \lt\ \intdkmj \lt \int_y^\tk h(x)\dx \rt^{m-1} \hspace{-6pt} h(y)\dy \rt^\pm \label{27}\\
      & \quad +  \sumk \ \intdkmd v(t) \Upm(t,\tkmj)\dt \lt\ \intdkmj \lt \int_y^\tk h(x)\dx \rt^{m-1} \hspace{-6pt} h(y)\dy \rt^\pm \nonumber \\
      & \quad +  \sumk \ \intdkmd v(t)\dt \lt\ \intdkmj U(\tkmj,y) \lt \int_y^\tk h(x)\dx \rt^{m-1} \hspace{-6pt} h(y)\dy \rt^\pm \nonumber\\
      & \lesssim \sumk \ \intdkmt v(t)\dt\ \Upm(\dkmd) \lt\ \intdkmj h(y)\dy \rt^p \label{28}\\
      & \quad +  \sumk \ \intdkmd v(t) \Upm(t,\tkmj)\dt \lt\ \intdkmj h(y)\dy \rt^p \nonumber \\
      & \quad +  \sumk \ \intdkmd v(t)\dt \lt\ \intdkmj u(s) \lt \int_s^\tk h(y)\dy \rt^{m} \ds \rt^\pm \nonumber\\
      & \lesssim \sumk \ \intdkmt v(t) \lt \int_t^\infty u(s) \lt \int_s^\infty h(y)\dy \rt^m \!\!\! \ds \rt^\pm \!\!\! \dt \nonumber\\
      & \quad +  \sumk \ \intdkmd v(t) \lt \int_t^\infty u(s) \lt \int_s^\infty h(y)\dy \rt^m \!\!\! \ds \rt^\pm \!\!\! \dt \nonumber\\
      & \lesssim \intnn v(t) \lt \inttn u(s) \lt \intsn h(y)\dy \rt^m \!\! \ds \rt^\pm \!\! \dt. \nonumber
    \end{align}
  Inequality \eqref{27} follows from \eqref{12}. The Fubini theorem is used to get the estimate of the third summand in \eqref{28}. The ``$\gtrsim$'' inequality in \eqref{19} is thus verified.
  
  Equivalence \eqref{117} is obtained by integration by parts, considering the a.e.-differentiability of $\fii$ (see Remark \ref{6}(i)). 
\end{proof}

The lemma below is a~standard auxillary result used in the theory of rearrangement-invariant spaces.

\begin{lem}\label{29}
  Let $m,p\in(0,\infty)$ and let $(u,v)$ be an~admissible pair of weights with respect to $(m,p)$. Then \eqref{30} holds for all $f\in\M$ if and only if the inequality
    \begin{equation}\label{31}
      \lt \intnn w(t) \lt \int_t^\infty h(y)\dy \rt^q \dt \rt^\jq \le C \lt \intnn v(t) \lt \inttn u(s) \lt \int_s^\infty h(y)\dy \rt^m \!\ds \rt^\pm \!\dt \rt^\jp
    \end{equation}
  holds for all $h\in\MM$.
\end{lem}

\begin{proof}
  The ``only if'' part is obvious. Let us prove the ``if'' part. Suppose that \eqref{31} holds for all $h\in\MM$. Let $f\in\M$. Then, by \cite[Lemma 1.2]{Si}, there exists a~sequence $\{h_n\}_{n\in\N}$ of functions from $\MM$ such that $\inttn h_n(y)\dy\uparrow \f(t)$ as $n\to\infty$ for a.e.~$t\in(0,\infty)$. For any $n\in\N$, \eqref{31} implies
    \begin{align*}
      \lt \intnn w(t) \lt \int_t^\infty h_n(y)\dy \rt^q \dt \rt^\jq & \le C \lt \intnn v(t) \lt \inttn u(s) \lt \int_s^\infty h_n(y)\dy \rt^m \!\! \ds \rt^\pm \!\! \dt \rt^\jp \\
                                                                    & \le C \lt \intnn v(t) \lt \inttn u(s) (\f(s))^m \ds \rt^\pm \!\! \dt \rt^\jp.
    \end{align*}
  The monotone convergence theorem then yields \eqref{30}. Since $f\in\M$ was chosen arbitrarily, \eqref{30} holds for all $f\in\M$.
\end{proof}

Since $(f^m)^*=(\f)^m$ is satisfied pointwise for every $f\in\M$ and $m>0$, the following proposition obviously holds true. 

\begin{prop}\label{33}
  Let $m,p\in(0,\infty)$ and let $(u,v)$ be an~admissible pair of weights with respect to $(m,p)$. Then inequality \eqref{30} holds with a~$C>0$ for all $f\in\M$ if and only if the inequality
    \begin{equation}  
      \lt \intnn w(t) (\f(t))^{\frac qm} \dt \rt^{\frac mq} \le C^m \lt \intnn v(t) \lt \inttn u(s) \f(s) \ds \rt^\pm \dt \rt^\mp
    \end{equation}
  holds for all $f\in\M$.
\end{prop}

\section{Embeddings $C\!L\hra \Lambda$}

In this section, the focus is laid on inequality \eqref{30}, whose validity for all $f\in\M$ corresponds to the existence of the embedding $\CL\hra\Lambda^q(w)$. In the light of Proposition \ref{33}, to study inequality $\eqref{30}$ it suffices to consider its rescaled version
  \begin{equation}\label{34}  
      \lt \intnn w(t) (\f(t))^q \dt \rt^\jq \le C \lt \intnn v(t) \lt \inttn u(s) \f(s) \ds \rt^p \dt \rt^\jp.
  \end{equation}
As before, convention \eqref{116} is used in all what follows.

\begin{thm}\label{32}
  Let $p,q\in(0,\infty)$, let $(u,v)$ be an~admissible pair of weights with respect to $(1,p)$ and let $w$ be a weight. 

  {\rm(i)}
        Let $1\le q<\infty$ and $0<p\le q$. Then \eqref{34} holds for all $f\in\M$ with a~constant $C>0$ independent of $f$ if and only if 
          \[
            A_1 := \sup_{t\in(0,\infty)} \lt \int_0^t w(s)\ds \rt^\jq \lt \int_0^t v(s) U^p(s,t) \ds \rt^\mjp  <\infty.
          \]
        Moreover, the optimal constant $C$ in \eqref{34} satisfies $C\approx A_1$.
  
  {\rm(ii)}
        Let $1\le q<p<\infty$ and $r:=\frac{pq}{p-q}$. Then \eqref{34} holds for all $f\in\M$ with a~constant $C>0$ independent of $f$ if and only if
          \[ 
            A_2 := \lt \intnn v(t) \sup_{y\in(t,\infty)} \frac{\Up(t,y) \lt \int_0^y w(s)\ds \rt^\rq }{ \lt \int_0^y v(x)\Up(x,y)\dx \rt^{\frac rq}} \dt \rt^\jr < \infty.
          \]
        Moreover, the optimal constant $C$ in \eqref{34} satisfies $C\approx A_2$.
  
  {\rm(iii)}
        Let $0<p\le q<1$. Then \eqref{34} holds for all $f\in\M$ with a~constant $C>0$ independent of $f$ if and only if 
          \[ 
            A_3 := \sup_{t>0} \lt \int_0^t v(s)\ds  \int_t^\infty \frac{ \lt \int_0^y w(x)\dx \rt^{1-q'} \! \Up(t,y) u(y) \int_0^y v(s) U^{p-1}(s,y) \ds}{ \lt \int_0^y v(z)\Up(z,y)\dz \rt^{2-\frac{q'}p}} \dy \rt^{-\frac1{q'}} \!\!\!\! < \infty,
          \]
          \[ 
            A_4 := \sup_{t>0} \lt \int_0^t v(s)\Up(s,t)\ds  \int_t^\infty \frac{ \lt \int_0^y w(x)\dx \rt^{1-q'} \!u(y) \int_0^y v(s) U^{p-1}(s,y) \ds}{ \lt \int_0^y v(z)\Up(z,y)\dz \rt^{2-\frac{q'}p}} \dy \rt^{-\frac1{q'}} \!\!\!\! < \infty
          \]
        and
          \[
            A_5 := \lt \intnn w(y)\dy \rt^\jq \lt \int_0^\infty v(s) U^p(s,\infty) \ds \rt^\mjp  < \infty.
          \]
        Moreover, the optimal constant $C$ in \eqref{34} satisfies $C\approx A_3 + A_4 + A_5$.
   
   {\rm(iv)}
        Let $0<q<1$, $q<p<\infty$ and $r:=\frac{pq}{p-q}$. Then \eqref{34} holds for all $f\in\M$ with a~constant $C>0$ independent of $f$ if and only if $A_5<\infty$ and
          \[
            A_6 := \lt \intnn v(t) \lt \int_t^\infty \frac{ \lt \int_0^y w(x)\dx \rt^{1-q'} U^\frac{p-q}{1-q}(t,y) u(y) \int_0^y v(s) U^{p-1}(s,y)\ds}{ \lt \int_0^y v(z)\Up(z,y)\dz \rt^{2-q'}} \dy \rt^{-\frac{r}{q'}} \!\!\! \dt \rt^\jr \!\!\!\! < \infty,
          \]  
        Moreover, the optimal constant $C$ in \eqref{34} satisfies $C\approx A_5 + A_6$.
\end{thm} 

\begin{proof}
  Let $\fii$ be the fundamental function of $\CLj$, and let $\{\tk\}_{k\in\K}$ be a~discretizing sequence of $\CLj$. Recall that the derivative $\fii'$ exists at a.e.~$t>0$ and is defined there by \eqref{44}. This explicit expression is used to state the $A$-conditions while the short notation $\fii'$ appears in the proof.
  
  \emph{Sufficiency.}
    Let $h\in\M$. By Theorem \ref{18} and Lemma \ref{29}, it suffices to prove that
      \begin{equation}\label{37}
        \lt \intnn w(t) \lt \int_t^\infty h(y)\dy \rt^q \dt \rt^\jq \le C \lt \sumk \lt\ \intdkmj \fii(y) h(y)\dy \rt^p \rt^\jp
      \end{equation}
    holds with some $C>0$. We have
      \Bdef{4}\Bdef{5}
      \begin{align*}
        & \intnn w(t) \lt \inttn h(y)\dy \rt^q \!\! \dt = \sumk\ \intdkmj w(t) \lt \inttn h(y)\dy \rt^q \!\! \dt \\
        & \quad \approx \sumk\ \intdkmj w(t) \lt \int_t^\tk h(y)\dy \rt^q \!\! \dt + \sumkmj\ \intdkmj w(t)\dt  \lt \int_\tk^\infty h(y)\dy \rt^q =: \B{4}+\B{5}.
      \end{align*} 
    At first, assume that $1\le q<\infty$. By the Hardy inequality (Proposition \ref{35}(i)), we get 
      \Bdef{6}
        \[  
          \B{4} \lesssim \sumk \lt\ \intdkmj \fii(y)h(y)\dy \rt^q \sup_{t\in\dkmj}\ \int_{\tkmj}^t w(s)\ds\ \fii^{-q}(t) =: \B{6}.
        \]
    
    (i) Let $0<p\le q$ and $A_1<\infty$. Then 
      \begin{align*}
        \B{6} & \le \sumk \lt\ \intdkmj \fii(y)h(y)\dy \rt^q \sup_{j\in\K} \sup_{t\in\djmj}\ \int_{\tjmj}^t w(s)\ds\ \fii^{-q}(t) \\  
              & \le \lt \sumk \lt\ \intdkmj \fii(y)h(y)\dy \rt^p \rt^\qp \sup_{j\in\K} \sup_{t\in\djmj}\ \int_{\tjmj}^t w(s)\ds\ \fii^{-q}(t) \\
              & \le \lt \sumk \lt\ \intdkmj \fii(y)h(y)\dy \rt^p \rt^\qp A^q_1.
      \end{align*}
    The second inequality above follows from convexity of the $\qp$-th power. Next, for $\B{5}$ we have
      \begin{align}
        \B{5} & \le      \sumkmj\, \fii^q(\tk) \lt \int_{\tk}^\infty h(y)\dy \rt^q \sup_{j\in\K}\, \intdjmj w(s)\ds\ \fii^{-q}(\tj) \nonumber\\
              & \lesssim \sumkmj\, \fii^q(\tk) \lt\ \intdk h(y)\dy \rt^q \sup_{j\in\K}\, \intdjmj w(s)\ds\ \fii^{-q}(\tj) \label{38}\\
              & \le      \lt \sumkmj\, \fii^p(\tk) \lt\ \intdk h(y)\dy \rt^p \rt^\qp \sup_{j\in\K}\, \intdjmj w(s)\ds\ \fii^{-q}(\tj) \label{39}\\
              & \le      \lt \sumkmj  \lt\ \intdk \fii(y) h(y)\dy \rt^p \rt^\qp A^q_1. \nonumber
      \end{align}
    In here, step \eqref{38} follows from Proposition \ref{20} and \eqref{3}, and step \eqref{39} follows from convexity of the $\qp$-th power. We have just proved that \eqref{37} holds with $C\lesssim A$.
    
    (ii) Let $q<p<\infty$ and $A_2<\infty$. Then the H\"older inequality (Proposition \ref{40}) implies 
      \[
        \B{6} \le \lt \sumk \lt\ \intdkmj \fii(y)h(y)\dy \rt^p \rt^\qp \lt \sum_{j\in\K} \, \sup_{t\in\djmj} \lt\ \int_{\tjmj}^t w(s)\ds \rt^\rq \fii^{-r}(t) \rt^\qr.
      \]
    Moreover, one has 
      \begin{align}
        \B{5} & \le      \sumkmj\, \fii^q(\tk) \lt \int_{\tk}^\infty h(y)\dy \rt^q \intdkmj w(s)\ds\ \fii^{-q}(\tk) \nonumber\\
              & \le      \lt \sumkmj\, \fii^p(\tk) \lt \int_{\tk}^\infty h(y)\dy \rt^p \rt^\qp \lt \sum_{j\in\K-1}\, \lt\ \intdjmj w(s)\ds \rt^\rq \fii^{-r}(\tj)  \rt^\qr \label{41}\\
              & \lesssim \lt \sumkmj\, \fii^p(\tk) \lt\ \intdk h(y)\dy \rt^p \rt^\qp \lt \sum_{j\in\K-1}\, \lt\ \intdjmj w(s)\ds \rt^\rq \fii^{-r}(\tj)  \rt^\qr \label{42}\\
              & \lesssim \lt \sumkmj \lt\ \intdk \fii(y) h(y)\dy \rt^p \rt^\qp \lt \sum_{j\in\K} \, \sup_{t\in\djmj} \lt\ \int_{\tjmj}^t w(s)\ds \rt^\rq \fii^{-r}(t) \rt^\qr. \nonumber
      \end{align}
    In \eqref{41} we used the H\"older inequality again. Estimate \eqref{42} follows from Proposition \ref{20} and \eqref{3}. To complete this part, we need the following estimate.
      \begin{align}
        & \sum_{j\in\K} \, \sup_{t\in\djmj} \lt\ \int_{\tjmj}^t w(s)\ds \rt^\rq \fii^{-r}(t)  \nonumber \\
        & =        \sum_{j\in\K} \, \sup_{t\in\djmj} \lt\ \int_{\tjmj}^t w(s)\ds \rt^\rq \!\fii^p(t) \fii^{-\frac{pr}q}(t) \nonumber \\
        & \lesssim \sum_{j\in\K} \ \intdjmt v(y)\dy\ \Up(\djmd) \sup_{t\in\djmj} \lt \int_0^t w(s)\ds \rt^\rq \! \fii^{-\frac{pr}q}(t) \label{43}\\
        & \quad +  \sum_{j\in\K} \ \intdjmd v(y)\Up(y,\tjmj)\dy \sup_{t\in\djmj} \lt \int_0^t w(s)\ds \rt^\rq \! \fii^{-\frac{pr}q}(t) \nonumber\\
        & \quad +  \sum_{j\in\K} \ \intdjmt v(y)\dy\ \sup_{t\in\djmj} \Up(\tjmj,t) \lt \int_0^t w(s)\ds \rt^\rq \! \fii^{-\frac{pr}q}(t) \nonumber\\
        & \lesssim \sum_{j\in\K} \ \intdjmj v(y) \sup_{t\in(y,\infty)} \Up(y,t) \lt \int_0^t w(s)\ds \rt^\rq \! \fii^{-\frac{pr}q}(t)\ \dy\nonumber\\
        & \le A^r_2. \nonumber
      \end{align}
    Inequality \eqref{43} follows from \eqref{12}. Summing up, we have proved that \eqref{37} holds with $C\lesssim A_2$.
    
    From now on, assume that $0<q<1$. Then 
      \Bdef{7}\Bdef{8}
        \begin{align}  
          \B{4} & \lesssim \sumk \lt\ \intdkmj \fii(y)h(y)\dy \rt^q \lt\ \intdkmj \lt\ \int_{\tkmj}^s w(y)\dy \rt^{-q'} \hspace{-6pt} w(s) \fii^{q'}(s) \ds \rt^{1-q}  \label{45}\\
                & \approx  \sumk \lt\ \intdkmj \fii(y)h(y)\dy \rt^q \intdkmj w(y)\dy\ \fii^{-q}(\tk)  \label{46}\\
                & \quad +  \sumk \lt\ \intdkmj \fii(y)h(y)\dy \rt^q \lt\ \intdkmj \lt\ \int_{\tkmj}^s w(y)\dy \rt^{1-q'} \hspace{-6pt} \fii^{q'-1}(s)\fii'(s) \ds \rt^{1-q} \nonumber \\
                &  =: \B{7}+\B{8}. \nonumber
        \end{align}
    Inequality \eqref{45} follows from the appropriate version of the Hardy inequality (Proposition \ref{35}(ii)). In step \eqref{46} we used integration by parts (notice the existence of $\fii'$ a.e.~on $(0,\infty)$ and the fact that $q'<0$). We will proceed by estimating the terms $\B{7}$, $\B{8}$ and $\B{5}$.
        
    (iii) Let $0<p\le q$ and $A_3 + A_4 + A_5<\infty$. It is easily verified that
      \begin{equation}\label{48}
        \fii^{-\alpha}(t) \approx \fii^{-\alpha}(\infty) + \int_t^\infty \fii^{-\alpha-1}(s) \fii'(s)\ds
      \end{equation}
    for any $t>0$ and any $\alpha>0$ (the constant in ``$\approx$'' depends on $\alpha$). By Proposition \ref{20} and \eqref{3}, one gets 
      \begin{align*}
        \B{5} & \le      \sumkmj\, \fii^q(\tk) \lt \int_{\tk}^\infty h(y)\dy \rt^q \sup_{j\in\K}\, \intdjmj w(s)\ds\ \fii^{-q}(\tj) \\
              & \lesssim \sumkmj\, \fii^q(\tk) \lt\ \intdk h(y)\dy \rt^q \sup_{j\in\K}\, \intdjmj w(s)\ds\ \fii^{-q}(\tj) \\
              & \le      \sumkmj  \lt\ \intdk \fii(y) h(y)\dy \rt^q \sup_{j\in\K}\, \intdjmj w(s)\ds\ \fii^{-q}(\tj). 
      \end{align*}
    Hence, using also convexity of the $\qp$-th power, we obtain
      \begin{align*}
        \B{5}+\B{7} & \lesssim \sumk \lt\ \intdkmj \fii(y)h(y)\dy \rt^q \sup_{j\in\K} \intdjmj w(y)\dy\ \fii^{-q}(\tj). \\
                    & \le      \lt \sumk \lt\ \intdkmj \fii(y)h(y)\dy \rt^p \rt^\qp \sup_{j\in\K} \intdjmj w(y)\dy\ \fii^{-q}(\tj). 
      \end{align*}
    Let us proceed as follows.
      \begin{align}
        & \sup_{j\in\K} \intdjmj w(y)\dy\ \fii^{-q}(\tj) \nonumber \\
                    & =        \sup_{j\in\K} \intdjmj w(y)\dy\ \lt\fii^p(\tj) \fii^{q'-p}(\tj)\rt^{-\frac{q}{q'}} \nonumber \\
                    & \lesssim \intnn w(y)\dy\ \fii^{-q}(\infty) +  \sup_{j\in\K} \intdjmj w(y)\dy\ \fii^{-\frac{pq}{q'}}(\tj) \lt \int_{\tj}^\infty \fii^{q'-p-1}(x)\fii'(x)\dx \rt^{-\frac{q}{q'}}  \label{50}\\
                    & \le      A_5^q + A_4^q. \nonumber 
      \end{align}
    Estimate \eqref{50} is a~consequence of \eqref{48} (notice that $-\frac{q}{q'}=\frac1{1-q'}$). Next, from convexity of the $\qp$-th power we get
      \begin{align*}
        \B{8} & \le      \sumk \lt\ \intdkmj \fii(y)h(y)\dy \rt^q \lt \,\sup_{j\in\K} \intdjmj \lt\ \int_{\tjmj}^s w(y)\dy \rt^{1-q'} \hspace{-6pt} \fii^{q'-1}(s)\fii'(s) \ds \rt^{-\frac{q}{q'}} \\  
              & \le      \lt \sumk \lt\ \intdkmj \fii(y)h(y)\dy \rt^p \rt^\qp \lt \,\sup_{j\in\K} \intdjmj \lt\ \int_{\tjmj}^s w(y)\dy \rt^{1-q'} \hspace{-6pt} \fii^{q'-1}(s)\fii'(s) \ds \rt^{-\frac{q}{q'}}.
      \end{align*}
    Furthermore, the following estimate is valid:
      \begin{align}
              & \sup_{j\in\K} \intdjmj \lt\ \int_{\tjmj}^s w(y)\dy \rt^{1-q'} \hspace{-6pt} \fii^{q'-1}(s)\fii'(s) \ds \nonumber\\
              & =        \sup_{j\in\K} \intdjmj \lt\ \int_{\tjmj}^s w(y)\dy \rt^{1-q'} \hspace{-6pt} \fii^p(s)\fii^{q'-p-1}(s)\fii'(s) \ds \nonumber\\
              & \lesssim \sup_{j\in\K} \intdjmt v(t)\dt\ \Up(\djmd) \intdjmj \lt\ \int_{\tjmj}^s w(y)\dy \rt^{1-q'} \hspace{-6pt} \fii^{q'-p-1}(s)\fii'(s) \ds \label{47}\\
              & \quad +  \sup_{j\in\K} \intdjmd v(t)\Up(t,\tjmj) \dt \intdjmj \lt\ \int_{\tjmj}^s w(y)\dy \rt^{1-q'} \hspace{-6pt} \fii^{q'-p-1}(s)\fii'(s) \ds \nonumber\\
              & \quad +  \sup_{j\in\K} \intdjmd v(t)\dt \intdjmj \lt\ \int_{\tjmj}^s w(y)\dy \rt^{1-q'} \hspace{-6pt} \Up(\tjmj,s) \fii^{q'-p-1}(s)\fii'(s) \ds \nonumber\\
              & \lesssim A_4^q + A_3^q. \nonumber
      \end{align}
    In here, \eqref{47} follows from \eqref{12}. By now, we have verified that \eqref{37} holds with $C\lesssim A_3+A_4+A_5$.

    (iv) Let $q<p<\infty$ and $A_5+A_6+A_7<\infty$. We have
      \begin{align}
        \B{5} & \le      \lt \sumkmj \fii^q(\tk) \lt\, \int_{\tk}^\infty h(y)\dy \rt^p \rt^\qp \lt \sum_{j\in\K} \lt\ \intdjmj w(s)\ds \rt^\rq \fii^{-r}(\tj) \rt^\qr \label{54}\\
              & \lesssim \lt \sumkmj \fii^q(\tk) \lt\ \intdk h(y)\dy \rt^p \rt^\qp \lt \sum_{j\in\K} \lt\ \intdjmj w(s)\ds \rt^\rq \fii^{-r}(\tj) \rt^\qr \label{55}\\
              & \le      \lt \sumkmj \lt\ \intdk \fii(y) h(y)\dy \rt^p \rt^\qp \lt \sum_{j\in\K} \lt\ \intdjmj w(s)\ds \rt^\rq \fii^{-r}(\tj) \rt^\qr \nonumber. 
      \end{align} 
    Step \eqref{54} follows from the H\"{o}lder inequality (Proposition \ref{40}) while in step \eqref{55} we used Proposition \ref{20} and \eqref{3}. The H\"{o}lder inequality also implies
      \[ 
        \B{7} \le \lt \sumkmj \lt\ \intdk \fii(y) h(y)\dy \rt^p \rt^\qp \lt \sum_{j\in\K} \lt\ \intdjmj w(s)\ds \rt^\rq \fii^{-r}(\tj) \rt^\qr.
      \]
    Moreover, one has
      \begin{align}
        & \sum_{j\in\K} \lt\ \intdjmj w(s)\ds \rt^\rq \fii^{-r}(\tj) \nonumber\\
        & =         \sum_{j\in\K} \lt\ \intdjmj w(s)\ds \rt^\rq \lt \fii^{-\frac{pq'}r}(\tj) \fii^{pq'-p}(\tj) \rt^{-\frac r{q'}} \nonumber\\
        & \lesssim  \sum_{j\in\K} \lt\ \intdjmj w(s)\ds \rt^\rq \fii^{-r}(\infty) + \sum_{j\in\K-1} \lt\ \intdjmj w(s)\ds \rt^\rq \fii^p(\tj) \lt \int_\tj^\infty \fii^{pq'-p-1}(x)\fii'(x)\dx \rt^{-\frac{r}{q'}} \label{56}\\
        & \lesssim  A_5^r + \sum_{j\in\K-1} \fii^p(\tj) \lt \int_\tj^\infty \lt \int_0^x  w(s)\ds \rt^{1-q'} \!\!\! \fii^{pq'-p-1}(x)\fii'(x)\dx \rt^{-\frac{r}{q'}} \label{57}\\
        & \lesssim  A_5^r + \sum_{j\in\K-1} \, \intdjmd v(y)\dy\ \Up(\djmj) \lt \int_\tj^\infty \lt \int_0^x  w(s)\ds \rt^{1-q'} \!\!\! \fii^{pq'-p-1}(x)\fii'(x)\dx \rt^{-\frac{r}{q'}} \label{58}\\
        & \quad +   \sum_{j\in\K-1} \, \intdjmj v(y)\Up(y,\tj)\dy\ \lt \int_\tj^\infty \lt \int_0^x  w(s)\ds \rt^{1-q'} \!\!\! \fii^{pq'-p-1}(x)\fii'(x)\dx \rt^{-\frac{r}{q'}}\nonumber \\
        & \lesssim  A_5^r + \sum_{j\in\K} \, \intdjmj v(y) \lt \int_y^\infty \lt \int_0^x  w(s)\ds \rt^{1-q'} \!\!\!U^\frac{p-q}{1-q}(y,x) \fii^{pq'-p-1}(x)\fii'(x)\dx \rt^{-\frac{r}{q'}} \nonumber\\
        & \lesssim  A_5^r + A_6^r. \nonumber
      \end{align}
    Inequality \eqref{56} follows from \eqref{48}, in step \eqref{57} we made use of convexity of the $\rq$-th power, and estimate \eqref{58} follows from \eqref{12} (by setting $t:=\tj\in\dj$). So far we have got
      \[
        \B{5}+\B{7} \lesssim \lt \sumkmj \lt\ \intdk \fii(y) h(y)\dy \rt^p \rt^\qp (A_5 + A_6)^q. 
      \]
    Next, the H\"older inequality (Proposition \ref{40}) yields
      \[
        \B{8} \le \lt \sumk \lt\ \intdkmj \fii(y)h(y)\dy \rt^p \rt^\qp \lt \sum_{j\in\K} \lt\ \intdjmj \lt\ \int_{\tjmj}^s w(y)\dy \rt^{1-q'} \hspace{-6pt} \fii^{q'-1}(s)\fii'(s) \ds \rt^{-\frac{r}{q'}} \rt^\qr. 
      \]
    Let us continue as follows.
      \begin{align}        
        & \sum_{j\in\K} \lt\ \intdjmj \lt\ \int_{\tjmj}^s w(y)\dy \rt^{1-q'} \hspace{-6pt} \fii^{q'-1}(s)\fii'(s) \ds \rt^{-\frac{r}{q'}} \nonumber\\
        & =        \sum_{j\in\K} \lt\ \intdjmj \lt\ \int_{\tjmj}^s w(y)\dy \rt^{1-q'} \hspace{-6pt} \fii^{-\frac{pq'}r}(s) \fii^{pq'-p-1}(s)\fii'(s) \ds \rt^{-\frac{r}{q'}} \nonumber\\
        & \lesssim \sum_{j\in\K}\ \intdjmt v(x)\dx\ \Up(\djmd) \lt\ \intdjmj \lt\ \int_{\tjmj}^s w(y)\dy \rt^{1-q'} \hspace{-6pt} \fii^{pq'-p-1}(s)\fii'(s) \ds \rt^{-\frac{r}{q'}} \label{60}\\
        & \quad +  \sum_{j\in\K}\ \intdjmd v(x)\Up(x,\tjmj) \dx\ \lt\ \intdjmj \lt\ \int_{\tjmj}^s w(y)\dy \rt^{1-q'} \hspace{-6pt} \fii^{pq'-p-1}(s)\fii'(s) \ds \rt^{-\frac{r}{q'}} \nonumber\\
        & \quad +  \sum_{j\in\K}\ \intdjmd v(x)\dx\ \lt\ \intdjmj \lt\ \int_{\tjmj}^s w(y)\dy \rt^{1-q'} \hspace{-6pt} U^\frac{p-q}{1-q}(\tjmj,s) \fii^{pq'-p-1}(s)\fii'(s) \ds \rt^{-\frac{r}{q'}} \nonumber\\
        & \lesssim \sum_{j\in\K}\ \intdjmj v(x)\dx\ \lt\ \int_x^\infty \lt\ \int_{x}^s w(y)\dy \rt^{1-q'} \hspace{-6pt} U^\frac{p-q}{1-q}(x,s) \fii^{pq'-p-1}(s)\fii'(s) \ds \rt^{-\frac{r}{q'}} \nonumber\\
        & = A_6^r. \nonumber
      \end{align}
    In here, we used \eqref{12} to get \eqref{60}. By combining all the obtained estimates, we verify that \eqref{37} holds with $C\lesssim A_5 + A_6$. This also completes the whole sufficiency part. 

    \emph{Necessity.} Assume that $p,q\in(0,\infty)$ and \eqref{34} holds with a~$C\in(0,\infty)$ for all $f\in\M$.     
    
    (i) Let $E\subset\R^n$ be a~set of measure $x\in(0,\infty)$ and define $f:=\chi_E$. Then $\f=\chi_{[0,x)}$. Since \eqref{34} holds with this particular $f$, we have
	    \[
		    \lt \int_0^x w(t)\dt \rt^\jq \lt \int_0^x v(t) U^p(t,x) \dt \rt^\jp \le C.
		\]
	Taking the supremum over $x>0$, we get
		\[
			A_1 \le C.
		\]
	This inequality holds for any choice of positive parameters $p,$ $q$. In particular, we get necessity of $A_1$ in case (i). 
    
    (ii) Let $1\le q<p<\infty$. Then
      \Bdef{9}\Bdef{10}\Bdef{11}
      \begin{align}
        A_2^r & =         \sum_{k\in\K} \ \intdkmj v(y) \sup_{t\in(y,\infty)} \Up(y,t) \lt \int_0^t w(s)\ds \rt^\rq \! \fii^{-\frac{pr}q}(t)\, \dy \nonumber\\
              & \lesssim  \sum_{k\in\K} \ \intdkmj v(y) \sup_{t\in(y,\tk)} \Up(y,t) \lt \int_0^t w(s)\ds \rt^\rq \! \fii^{-\frac{pr}q}(t)\, \dy \nonumber\\
              & \quad +   \sum_{k\in\K-1} \ \intdkmj v(y) \Up(y,\tk) \dy \sup_{t\in(\tk,\infty)} \lt \int_0^t w(s)\ds \rt^\rq \! \fii^{-\frac{pr}q}(t) \nonumber\\
              & \quad +   \sum_{k\in\K-1} \ \intdkmj v(y) \dy \sup_{t\in(\tk,\infty)} \Up(\tk,t) \lt \int_0^t w(s)\ds \rt^\rq \! \fii^{-\frac{pr}q}(t) \nonumber\\
              & =: \B{9}+\B{10}+\B{11}.\nonumber
      \end{align}
    The estimating continues as follows.
      \begin{align}
        \B{9} & =         \sum_{k\in\Kj} \ \intdkmj v(y) \sup_{t\in(y,\tk)} \Up(y,t) \lt \int_0^t w(s)\ds \rt^\rq \! \fii^{-\frac{pr}q}(t)\, \dy \nonumber\\
              & \quad +   \sum_{k\in\Kd} \ \intdkmj v(y) \sup_{t\in(y,\tk)} \Up(y,t) \lt \int_0^t w(s)\ds \rt^\rq \! \fii^{-\frac{pr}q}(t)\, \dy \nonumber\\
              & \le       \sum_{k\in\Kj} \ \int_0^\tk v(y) \dy \ \sup_{t\in\dk} \Up(\tkmj,t) \lt \int_0^t w(s)\ds \rt^\rq \! \fii^{-\frac{pr}q}(t) \nonumber\\
              & \quad +   \sum_{k\in\Kd} \ \fii^p(\tk) \sup_{t\in\dk} \lt \int_0^t w(s)\ds \rt^\rq \! \fii^{-\frac{pr}q}(t) \nonumber\\
              & \approx   \sum_{k\in\Kj} \ \int_0^\tkmj v(y) \dy \ \sup_{t\in\dk} \Up(\tkmj,t) \lt \int_0^t w(s)\ds \rt^\rq \! \fii^{-\frac{pr}q}(t) \label{64}\\
              & \quad +   \sum_{k\in\Kd} \ \fii^p(\tkmj) \sup_{t\in\dk} \lt \int_0^t w(s)\ds \rt^\rq \! \fii^{-\frac{pr}q}(t) \nonumber\\
              & \lesssim  \sum_{k\in\K} \ \sup_{t\in\dk} \lt \int_0^t w(s)\ds \rt^\rq \! \fii^{-r}(t). \nonumber
      \end{align}
    To get \eqref{64}, relations \eqref{4} and \eqref{5} were used. Next, we get
      \begin{align}
        \B{10} & \le      \sum_{k\in\K-1} \fii^p(\tk) \sup_{t\in(\tk,\infty)} \lt \int_0^t w(s)\ds \rt^\rq \! \fii^{-\frac{pr}q}(t) \nonumber\\
               & \lesssim \sum_{k\in\K-1} \fii^p(\tk) \sup_{t\in\dk} \lt \int_0^t w(s)\ds \rt^\rq \! \fii^{-\frac{pr}q}(t) \label{65}\\
               & \le      \sum_{k\in\K}\ \sup_{t\in\dk} \lt \int_0^t w(s)\ds \rt^\rq \! \fii^{-r}(t). \nonumber 
      \end{align}
    Step \eqref{65} makes use of Proposition \ref{20} and \eqref{3}. Furthermore, one has
      \begin{align}
        \B{11} & \le      \sum_{k\in\K-1} \ \intdkmj v(y) \dy \sum_{j=k}^{K-1} \,\sup_{t\in\dj} \Up(\tk,t) \lt \int_0^t w(s)\ds \rt^\rq \! \fii^{-\frac{pr}q}(t) \nonumber\\
               & \lesssim \sum_{k\in\K-2} \ \intdkmj v(y) \dy \sum_{j=k+1}^{K-1} \Up(\tk,\tj) \, \sup_{t\in\dj} \lt \int_0^t w(s)\ds \rt^\rq \! \fii^{-\frac{pr}q}(t) \nonumber\\
               & \quad +  \sum_{k\in\K-1}  \int_0^\tk v(y) \dy\ \sum_{j=k}^{K-1} \, \sup_{t\in\dj} \Up(\tj,t) \lt \int_0^t w(s)\ds \rt^\rq \! \fii^{-\frac{pr}q}(t) \nonumber\\
               & \lesssim \sum_{j\in\K-1} \ \sup_{t\in\dj} \lt \int_0^t w(s)\ds \rt^\rq \! \fii^{-\frac{pr}q}(t) \sum_{k=-\infty}^{j-1}\,\intdkmj v(y) \dy \ \Up(\tk,\tj)  \label{66}\\
               & \quad +  \sum_{k\in\K-1}  \int_0^\tk v(y) \dy \ \sup_{t\in\dk} \Up(\tk,t) \lt \int_0^t w(s)\ds \rt^\rq \! \fii^{-\frac{pr}q}(t) \nonumber\\
               & \le      \sum_{j\in\K-1} \ \sup_{t\in\dj} \lt \int_0^t w(s)\ds \rt^\rq \! \fii^{-\frac{pr}q}(t) \fii^p(\tj)  \nonumber\\
               & \quad +  \sum_{k\in\K-1} \ \sup_{t\in\dk} \int_0^\tk v(y) \dy \ \Up(\tk,t) \lt \int_0^t w(s)\ds \rt^\rq \! \fii^{-\frac{pr}q}(t) \nonumber\\
               & \lesssim \sum_{k\in\K}\ \sup_{t\in\dk} \lt \int_0^t w(s)\ds \rt^\rq \! \fii^{-r}(t). \nonumber
      \end{align}
    The estimate of the second sum in \eqref{66} follows from Proposition \ref{20} with \eqref{2}. Moreover, the following estimate is valid.
      \begin{align}
        & \sum_{k\in\K}\ \sup_{t\in\dk} \lt \int_0^t w(s)\ds \rt^\rq \! \fii^{-r}(t) \nonumber\\
            & \approx  \sum_{k\in\K}\ \sup_{t\in\dkmj} \lt\ \int_\tkmj^t w(s)\ds \rt^\rq \fii^{-r}(t) + \sum_{k\in\K} \lt \int_0^\tkmj w(s)\ds \rt^\rq \fii^{-r}(\tkmj) \nonumber\\
            & \lesssim \sum_{k\in\K}\ \sup_{t\in\dkmj} \lt\ \int_\tkmj^t w(s)\ds \rt^\rq \fii^{-r}(t) + \sum_{k\in\K} \lt\ \intdkmj w(s)\ds \rt^\rq \fii^{-r}(\tkmj) \label{69}\\
            & \lesssim \sum_{k\in\K}\ \sup_{t\in\dkmj} \lt\ \int_\tkmj^t w(s)\ds \rt^\rq \fii^{-r}(t). \nonumber
      \end{align}
    Step \eqref{69} follows from Proposition \ref{21} with \eqref{3}. So far, we have proved
      \[
        A_2 \lesssim \lt \sum_{k\in\K}\ \sup_{t\in\dkmj} \lt\ \int_\tkmj^t w(s)\ds \rt^\rq \fii^{-r}(t)\rt^\jr.
      \]
    By saturation of the Hardy inequality (Proposition \ref{35}(i)), for each $k\in\K$ there exists a~function $g_k\in\MM$ supported in $\dkmj$ and such that $\intdkmj \fii(y)g_k(y)\dy=1$ and
    \begin{equation}\label{67}
	    \sup_{t\in\dkmj} \lt\ \int_\tkmj^t w(s)\ds \rt^\jq \fii^{-1}(t) \lesssim \lt\ \intdkmj \lt \int_t^\tk g_k(y)\dy \rt^q w(t)\dt \rt^\jq.
    \end{equation}
    Let $\{c_k\}_{k\in\K}$ be a~nonnegative sequence such that $\sumk c_k^p=1$ and 
      \begin{equation}\label{70}
        \lt \sum_{k\in\K}\ \sup_{t\in\dkmj} \lt\ \int_\tkmj^t w(s)\ds \rt^\rq \fii^{-r}(t)\rt^\jr \lesssim \lt \sum_{k\in\K}\ c_k^q \sup_{t\in\dkmj} \lt\ \int_\tkmj^t w(s)\ds \rt^\jq \fii^{-1}(t)\rt^\jq.
      \end{equation}
    Existence of such sequence is granted by saturation of the discrete H\"older inequality (Proposition \ref{40}). Let us define the function $g:=\sumk c_kg_k$. Then we have
      \begin{align}
        \lt \sum_{k\in\K}\ \sup_{t\in\dkmj} \lt\ \int_\tkmj^t w(s)\ds \rt^\rq \fii^{-r}(t)\rt^\jr & \lesssim \lt \sum_{k\in\K}\ c_k^q \sup_{t\in\dkmj} \lt\ \int_\tkmj^t w(s)\ds \rt^\jq \fii^{-1}(t)\rt^\jq \nonumber\\
                                                                                                  & \lesssim \lt \sum_{k\in\K}\ c_k^q\ \intdkmj \lt \int_t^\tk g_k(y)\dy \rt^q w(t)\dt \rt^\jq \label{71}\\
                                                                                                  & =        \lt \intnn \lt \int_t^\infty g(y)\dy \rt^q w(t)\dt \rt^\jq \nonumber\\
                                                                                                  & \le C    \lt \intnn v(t)\lt\int_t^\infty u(s) \int_s^\infty g(y)\dy\ds \rt^p \dt \rt^\jp \label{72}\\
                                                                                                  & \approx C\lt \sum_{k\in\K} \lt\ \intdkmj \fii(y)g(y) \dy \rt^p \rt^\jp \label{73}\\
                                                                                                  & \approx C\lt \sum_{k\in\K} c_k^p \lt\ \intdkmj \fii(y)g_k(y) \dy \rt^p \rt^\jp \nonumber\\
                                                                                                  & = C.\nonumber
      \end{align}
    In step \eqref{71} we used \eqref{67}, inequality \eqref{72} follows from \eqref{34}, and \eqref{73} follows from Theorem \ref{18}. We have now verified that
      \[
        A_2 \lesssim \lt \sum_{k\in\K}\ \sup_{t\in\dkmj} \lt\ \int_\tkmj^t w(s)\ds \rt^\rq \fii^{-r}(t)\rt^\jr \lesssim C.
      \]
    
    Throughout the next part, assume that $0<q<1$. Saturation of the Hardy inequality (Proposition \ref{35}(ii)) then guarantees that for each $k\in\K$ there exists a~function $\psi_k\in\MM$ supported in $\dkmj$ and such that $\intdkmj \fii(y)\psi_k(y)\dy=1$ and
      \begin{equation}\label{74}
        \lt\ \intdkmj \lt\ \int_\tkmj^t w(s)\ds \rt^{-q'} \!\! w(t) \fii^{-q'}(t) \dt \rt^{-\frac1{q'}} \lesssim \lt\ \intdkmj \lt \int_t^\tk \psi_k(y)\dy \rt^q w(t)\dt \rt^\jq.
      \end{equation}
      
    (iii) Let $0<p\le q$. Then
      \Bdef{12}\Bdef{13}\Bdef{14}
      \begin{align*}
        A_3^{-q'}& \approx  \sup_{k\in\K}\ \sup_{t\in\dkmj} \int_0^t v(s)\ds \int_t^\infty \lt \int_0^y w(x)\dx \rt^{1-q'} \!\!\! \Up(t,y) \fii^{q'-p-1}(y)\fii'(y)\dy \\
                 & \lesssim \sup_{k\in\K}\ \sup_{t\in\dkmj} \int_0^t v(s)\ds \int_t^\tk \lt \int_0^y w(x)\dx \rt^{1-q'} \!\!\! \Up(t,y) \fii^{q'-p-1}(y)\fii'(y)\dy \\
                 & \quad +  \sup_{k\in\K-1}\ \sup_{t\in\dkmj} \int_0^t v(s)\ds\ \Up(t,\tk) \int_\tk^\infty \lt \int_0^y w(x)\dx \rt^{1-q'} \!\!\! \fii^{q'-p-1}(y)\fii'(y)\dy \\
                 & \quad +  \sup_{k\in\K-1}\, \int_0^\tk v(s)\ds \int_\tk^\infty \lt \int_0^y w(x)\dx \rt^{1-q'} \!\!\! \Up(\tk,y) \fii^{q'-p-1}(y)\fii'(y)\dy \\
                 & =: \B{12}+\B{13}+\B{14}.
      \end{align*}
    One can show that
      \begin{equation}\label{78}  
        \sup_{k\in\K}\ \intdkmj \lt \int_0^y w(x)\dx \rt^{1-q'} \!\!\! \fii^{q'-1}(y)\fii'(y)\dy \lesssim \sup_{k\in\K}\, \intdkmj \lt\ \int_\tkmj^y w(x)\dx \rt^{-q'} \!\!\! w(y) \fii^{q'}(y)\dy.
      \end{equation}
    Indeed, there holds
      \begin{align}
        & \sup_{k\in\K}\ \intdkmj \lt \int_0^y w(x)\dx \rt^{1-q'} \!\!\! \fii^{q'-1}(y)\fii'(y)\dy \nonumber\\
        & \lesssim \sup_{k\in\K}\, \lt \int_0^\tkmj w(x)\dx \rt^{1-q'} \!\!\!\!\! \intdkmj \fii^{q'-1}(y)\fii'(y)\dy + \sup_{k\in\K}\, \intdkmj \lt\ \int_\tkmj^y w(x)\dx \rt^{1-q'} \hspace{-8pt}\fii^{q'-1}(y)\fii'(y)\dy \nonumber\\
        & \lesssim \sup_{k\in\K}\, \lt \int_0^\tkmj w(x)\dx \rt^{1-q'} \!\!\! \fii^{q'}(\tkmj) +  \sup_{k\in\K}\, \intdkmj \lt\ \int_\tkmj^y w(x)\dx \rt^{-q'} \!\!\! w(y) \fii^{q'}(y)\dy \label{76}\\
        & \lesssim \sup_{k\in\K}\, \lt\ \intdkmd w(x)\dx \rt^{1-q'} \!\!\! \fii^{q'}(\tkmj) +  \sup_{k\in\K}\, \intdkmj \lt\ \int_\tkmj^y w(x)\dx \rt^{-q'} \!\!\! w(y) \fii^{q'}(y)\dy \label{77}\\
        & \approx  \sup_{k\in\K}\, \intdkmd \lt\ \int_\tkmd^y w(x)\dx \rt^{-q'} \!\!\! w(y) \dy\ \fii^{q'}(\tkmj) +  \sup_{k\in\K}\, \intdkmj \lt\ \int_\tkmj^y w(x)\dx \rt^{-q'} \!\!\! w(y) \fii^{q'}(y)\dy \nonumber\\
        & \lesssim \sup_{k\in\K}\, \intdkmj \lt\ \int_\tkmj^y w(x)\dx \rt^{-q'} \!\!\! w(y) \fii^{q'}(y)\dy. \nonumber
      \end{align}
    Step \eqref{76} is based on integration by parts yields (notice that $\qqp<0$), and \eqref{77} follows from Proposition \ref{21} with \eqref{3}. Hence, \eqref{78} is proved. Now, from \eqref{78} it follows that
      \[
        \B{12}  \le \sup_{k\in\K}\, \intdkmj \!\lt \int_0^y w(x)\dx \rt^{1-q'} \!\!\!\! \fii^{q'-1}(y)\fii'(y)\dy \, \lesssim\, \sup_{k\in\K}\, \intdkmj \! \lt\ \int_\tkmj^y w(x)\dx \rt^{-q'} \!\!\! w(y) \fii^{q'}(y)\dy.
      \]
    Next, one has
      \begin{align}
        \B{13} & \le      \sup_{k\in\K-1}\ \fii^p(\tk) \int_\tk^\infty \lt \int_0^y w(x)\dx \rt^{1-q'} \!\!\! \fii^{q'-p-1}(y)\fii'(y)\dy \nonumber\\
               & \lesssim \sup_{k\in\K-1}\ \fii^p(\tk) \intdk \lt \int_0^y w(x)\dx \rt^{1-q'} \!\!\! \fii^{q'-p-1}(y)\fii'(y)\dy \label{79}\\
               & \le      \sup_{k\in\K-1}\ \intdk \lt \int_0^y w(x)\dx \rt^{1-q'} \!\!\! \fii^{q'-1}(y)\fii'(y)\dy \nonumber\\
               & \lesssim \sup_{k\in\K}\, \intdkmj \! \lt\ \int_\tkmj^y w(x)\dx \rt^{-q'} \!\!\! w(y) \fii^{q'}(y)\dy. \label{80}
      \end{align}
    Proposition \ref{20} and \eqref{3} were used to get \eqref{79}, and \eqref{78} was used in \eqref{80}. For the term $\B{14}$ one gets
      \Bdef{15}\Bdef{16}
      \begin{align}
        \B{14} & =        \sup_{k\in\K-1}\, \int_0^\tk v(s)\ds \sum_{j=k}^{K-1} \intdj \lt \int_0^y w(x)\dx \rt^{1-q'} \!\!\! \Up(\tk,y) \fii^{q'-p-1}(y)\fii'(y)\dy \nonumber\\
               & \lesssim \sup_{k\in\K-2}\, \int_0^\tk v(s)\ds \sum_{j=k+1}^{K-1} \Up(\tk,\tj) \intdj \lt \int_0^y w(x)\dx \rt^{1-q'} \!\!\! \fii^{q'-p-1}(y)\fii'(y)\dy \nonumber\\
               & \quad +  \sup_{k\in\K-1}\, \int_0^\tk v(s)\ds \sum_{j=k}^{K-1} \intdj \lt \int_0^y w(x)\dx \rt^{1-q'} \!\!\! \Up(\tj,y) \fii^{q'-p-1}(y)\fii'(y)\dy\nonumber\\
               & =: \B{15}+\B{16}. \nonumber
      \end{align}
    By iterating the inequality $(a+b)^p\le 2^p(a^p+b^p)$ valid for $a,b\ge 0$, it is verified that
      \begin{equation}\label{81}
        \Up(\tk,\tj) \le \sum_{i=k}^{j-1} 2^{p(i-k+1)}\Up(\Delta_i) = 2^{-pk}\sum_{i=k}^{j-1} 2^{p(i+1)}\Up(\Delta_i)
      \end{equation}
    holds for any $j,k\in\K$ such that $k<j$. Then
      \begin{align}
        \B{15} & \le       \sup_{k\in\K-2}\, 2^{-pk} \int_0^\tk v(s)\ds \sum_{j=k+1}^{K-1} \sum_{i=k}^{j-1} 2^{p(i+1)}\Up(\Delta_i) \intdj \lt \int_0^y w(x)\dx \rt^{1-q'} \!\!\! \fii^{q'-p-1}(y)\fii'(y)\dy \label{82}\\
               & =         \sup_{k\in\K-2}\, 2^{-pk} \int_0^\tk v(s)\ds \sum_{i=k}^{K-2} 2^{p(i+1)} \Up(\Delta_i) \sum_{j=i+1}^{K-1} \intdj \lt \int_0^y w(x)\dx \rt^{1-q'} \!\!\! \fii^{q'-p-1}(y)\fii'(y)\dy \nonumber\\
               & =         \sup_{k\in\K-2}\, 2^{-pk} \int_0^\tk v(s)\ds \sum_{i=k}^{K-2} 2^{p(i+1)} \Up(\Delta_i) \int_{t_{i+1}}^\infty \lt \int_0^y w(x)\dx \rt^{1-q'} \!\!\! \fii^{q'-p-1}(y)\fii'(y)\dy \nonumber\\
               & \lesssim  \sup_{k\in\K-2}\, \int_0^\tk v(s)\ds \ \Up(\dk) \int_{t_{k+1}}^\infty \lt \int_0^y w(x)\dx \rt^{1-q'} \!\!\! \fii^{q'-p-1}(y)\fii'(y)\dy \label{83}\\
               & \le       \sup_{k\in\K-2}\, \fii^p(t_{k+1}) \int_{t_{k+1}}^\infty \lt \int_0^y w(x)\dx \rt^{1-q'} \!\!\! \fii^{q'-p-1}(y)\fii'(y)\dy \nonumber\\
               & \le       \sup_{k\in\K-2}\, \fii^p(t_{k+1}) \int_{\Delta_{k+1}} \lt \int_0^y w(x)\dx \rt^{1-q'} \!\!\! \fii^{q'-p-1}(y)\fii'(y)\dy \label{84}\\
               & \lesssim  \sup_{k\in\K}\, \int_{\Delta_{k-1}} \lt \int_0^y w(x)\dx \rt^{1-q'} \!\!\! \fii^{q'-1}(y)\fii'(y)\dy \nonumber\\
               & \lesssim  \sup_{k\in\K}\, \intdkmj \lt\ \int_\tkmj^y w(x)\dx \rt^{-q'} \!\!\! w(y) \fii^{q'}(y)\dy. \label{85}
      \end{align}
    In step \eqref{82} we used \eqref{81}. Inequality \eqref{83} follows from Proposition \ref{20} since 
      \begin{equation}\label{104}
        2\cdot 2^{-p(k-1)} \int_0^\tkmj v(s)\ds \le 2^{-pk} \int_0^\tk v(s)\ds
      \end{equation}
    for all $k\in\K$ thanks to \eqref{2}. Proposition \ref{20} and \eqref{3} also yield inequality \eqref{84}. The last step \eqref{85} is due to \eqref{78}. Next, the term $\B{16}$ is treated as follows.
      \begin{align}
        \B{16} & \lesssim \sup_{k\in\K-1}\, \int_0^\tk v(s)\ds \ \intdk \lt \int_0^y w(x)\dx \rt^{1-q'} \!\!\! \Up(\tk,y) \fii^{q'-p-1}(y)\fii'(y)\dy \label{86}\\
               & \le      \sup_{k\in\K-1}\, \intdk \lt \int_0^y w(x)\dx \rt^{1-q'} \!\!\! \Up(\tk,y) \fii^{q'-1}(y)\fii'(y)\dy \nonumber\\
               & \lesssim \sup_{k\in\K}\, \intdkmj \lt\ \int_\tkmj^y w(x)\dx \rt^{-q'} \!\!\! w(y) \fii^{q'}(y)\dy. \label{87}
      \end{align}
    In \eqref{86} we used Proposition \ref{20} and \eqref{2}, and in \eqref{87} we used \eqref{78}. At this point, we have proved the estimate
      \[
        A_3 \lesssim \lt \sup_{k\in\K}\, \intdkmj \lt\ \int_\tkmj^y w(x)\dx \rt^{-q'} \!\!\! w(y) \fii^{q'}(y)\dy \rt^{-\frac1{q'}}.
      \]
    The estimating now continues with the term $A_4$. One has
      \begin{align}
        A_4^{-q'} & \approx  \sup_{k\in\K}\ \sup_{t\in\dkmj} \fii^p(t) \int_t^\infty \lt \int_0^y w(x)\dx \rt^{1-q'} \!\!\! \fii^{q'-p-1}(y)\fii'(y)\dy \nonumber\\
                  & \lesssim \sup_{k\in\K}\ \sup_{t\in\dkmj} \fii^p(t) \int_t^\tk \lt \int_0^y w(x)\dx \rt^{1-q'} \!\!\! \fii^{q'-p-1}(y)\fii'(y)\dy \nonumber\\
                  & \quad +  \sup_{k\in\K-1}\ \fii^p(\tk) \int_\tk^\infty \lt \int_0^y w(x)\dx \rt^{1-q'} \!\!\! \fii^{q'-p-1}(y)\fii'(y)\dy \nonumber\\
                  & \lesssim \sup_{k\in\K}\ \intdk \lt \int_0^y w(x)\dx \rt^{1-q'} \!\!\! \fii^{q'-1}(y)\fii'(y)\dy \label{91}\\
                  & \quad +  \sup_{k\in\K-1}\ \fii^p(\tk) \intdk \lt \int_0^y w(x)\dx \rt^{1-q'} \!\!\! \fii^{q'-p-1}(y)\fii'(y)\dy \nonumber\\
                  & \lesssim \sup_{k\in\K}\ \intdk \lt \int_0^y w(x)\dx \rt^{1-q'} \!\!\! \fii^{q'-1}(y)\fii'(y)\dy \nonumber\\
                  & \lesssim \sup_{k\in\K}\, \intdkmj \lt\ \int_\tkmj^y w(x)\dx \rt^{-q'} \!\!\! w(y) \fii^{q'}(y)\dy. \label{92}
      \end{align}
    Inequality \eqref{91} follows from Proposition \ref{20} with \eqref{3}, and inequality \eqref{92} was proved in \eqref{78}. Hence,
      \[
        A_4 \lesssim \lt \sup_{k\in\K}\, \intdkmj \lt\ \int_\tkmj^y w(x)\dx \rt^{-q'} \!\!\! w(y) \fii^{q'}(y)\dy \rt^{-\frac1{q'}}.
      \]
    Next, the term $A_5$ can be estimated as follows.
      \begin{align}
        A_5 & \le      \sup_{k\in\K} \sup_{t\in\dkmj} \lt \int_0^t w(s)\ds \rt^\jq \fii^{-1}(t) \nonumber\\
            & \approx  \sup_{k\in\K} \sup_{t\in\dkmj} \lt\ \int_\tkmj^t w(s)\ds \rt^\jq \fii^{-1}(t) + \sup_{k\in\K} \lt \int_0^\tkmj w(s)\ds \rt^\jq \fii^{-1}(\tkmj) \nonumber\\
            & \lesssim \sup_{k\in\K} \sup_{t\in\dkmj} \lt\ \int_\tkmj^t w(s)\ds \rt^\jq \fii^{-1}(t) + \sup_{k\in\K} \lt\ \intdkmd w(s)\ds \rt^\jq \fii^{-1}(\tkmj) \label{93}\\
            & \lesssim \sup_{k\in\K} \sup_{t\in\dkmj} \lt\ \int_\tkmj^t w(s)\ds \rt^\jq \fii^{-1}(t) \nonumber \\
            & \lesssim \sup_{k\in\K} \sup_{t\in\dkmj} \lt\ \int_\tkmj^t w(s)\ds \rt^\jq \lt \int_t^\tk \fii^{q'-1}(y)\fii'(y) \dy \rt^{-\frac1{q'}}\nonumber \\
            & \lesssim \lt \sup_{k\in\K}\ \intdkmj \lt \int_0^y w(x)\dx \rt^{1-q'} \!\!\! \fii^{q'-1}(y)\fii'(y)\dy \rt^{-\frac1{q'}} \nonumber\\
            & \lesssim \lt \sup_{k\in\K}\, \intdkmj \lt\ \int_\tkmj^y w(x)\dx \rt^{-q'} \!\!\! w(y) \fii^{q'}(y)\dy \rt^{-\frac1{q'}}. \label{94}
      \end{align}
    Step \eqref{93} follows from Proposition \ref{21} and \eqref{3}, and step \eqref{94} from \eqref{78}. The estimate
      \[
        A_5 \lesssim \lt \sup_{k\in\K}\, \intdkmj \lt\ \int_\tkmj^y w(x)\dx \rt^{-q'} \!\!\! w(y) \fii^{q'}(y)\dy \rt^{-\frac1{q'}}
      \]
    is thus proven. Let $k\in\K$ be fixed and recall the definition of the function $\psi_{k}$. One has
      \begin{align}
        \lt\ \intdkmj \lt\ \int_\tkmj^t w(s)\ds \rt^{-q'} \!\! w(t) \fii^{-q'}(t) \dt \rt^{-\frac1{q'}} & \lesssim \lt\ \intdkmj \lt \int_t^\tk \psi_k(y)\dy \rt^q w(t)\dt \rt^\jq \label{88}\\
                                                                                                  & =        \lt \intnn \lt \int_t^\infty \psi_k(y)\dy \rt^q w(t)\dt \rt^\jq \nonumber\\
                                                                                                  & \le C    \lt \intnn v(t)\lt\int_t^\infty u(s) \int_s^\infty \psi_k(y)\dy\ds \rt^p \dt \rt^\jp \label{89}\\
                                                                                                  & \approx C\lt \sum_{j\in\K} \lt\ \intdjmj \fii(y)\psi_k(y) \dy \rt^p \rt^\jp \label{90}\\
                                                                                                  & = C \intdkmj \fii(y)\psi_k(y) \dy \nonumber\\
                                                                                                  & = C.\nonumber
      \end{align}
    Step \eqref{88} follows from \eqref{74}, step \eqref{89} follows from \eqref{34}, and step \eqref{90} from Theorem \ref{18}. Finally, we have proved
      \[
        A_3 + A_4 + A_5 \lesssim \lt \sup_{k\in\K}\, \intdkmj \lt\ \int_\tkmj^y w(x)\dx \rt^{-q'} \!\!\! w(y) \fii^{q'}(y)\dy \rt^{-\frac1{q'}} \!\!\lesssim C.
      \]
   
    (iv) Let $q<p<\infty$. Then
      \Bdef{17}\Bdef{18}\Bdef{19}
      \begin{align}
        A_6^r & \approx  \sumk\ \intdkmj v(t) \lt \int_t^\infty \lt \int_0^y w(x)\dx \rt^{1-q'} \hspace{-4pt} U^{\frac{p-q}{1-q}}(t,y) \fii^{pq'-p-1}(y)\fii'(y)\dy \rt^{-\frac r{q'}} \hspace{-4pt} \dt \nonumber\\
              & \lesssim \sumk\ \intdkmj v(t) \lt \int_t^\tk \lt \int_0^y w(x)\dx \rt^{1-q'} \hspace{-4pt} U^{\frac{p-q}{1-q}}(t,y) \fii^{pq'-p-1}(y)\fii'(y)\dy \rt^{-\frac r{q'}} \hspace{-4pt} \dt \nonumber\\
              & \quad +  \sum_{k\in\K-1}\ \intdkmj v(t)\Up(t,\tk)\dt\ \lt\, \int_\tk^\infty \lt \int_0^y w(x)\dx \rt^{1-q'} \hspace{-4pt} \fii^{pq'-p-1}(y)\fii'(y)\dy \rt^{-\frac r{q'}} \nonumber\\
              & \quad +  \sum_{k\in\K-1}\ \intdkmj v(t)\dt\ \lt \int_\tk^\infty \lt \int_0^y w(x)\dx \rt^{1-q'} \hspace{-4pt} U^{\frac{p-q}{1-q}}(\tk,y) \fii^{pq'-p-1}(y)\fii'(y)\dy \rt^{-\frac r{q'}} \nonumber\\
              & =: \B{17}+\B{18}+\B{19}. \nonumber        
      \end{align}
    Observe that
      \begin{equation}\label{95}  
        \sum_{k\in\K} \lt\ \intdkmj \! \lt \int_0^y w(x)\dx \rt^{1-q'} \!\!\!\!\!\! \fii^{q'-1}(y)\fii'(y)\dy \rt^{-\frac r{q'}} \!\!\!\!\! \lesssim \sum_{k\in\K} \lt\ \intdkmj \!\! \lt\ \int_\tkmj^y w(x)\dx \rt^{-q'} \!\!\!\!\! w(y) \fii^{q'}(y)\dy \rt^{-\frac r{q'}}\!\!\!\!.
      \end{equation}
    This relation is proved analogously as \eqref{78} was proved earlier, replacing the supremum in \eqref{78} by the sum and using an~appropriate version of Proposition \ref{21}. Next, we have
      \begin{align}
        \B{17} & =        \sum_{k\in\Kj}\ \intdkmj v(t) \lt \int_t^\tk \lt \int_0^y w(x)\dx \rt^{1-q'} \hspace{-4pt} U^{\frac{p-q}{1-q}}(t,y) \fii^{pq'-p-1}(y)\fii'(y)\dy \rt^{-\frac r{q'}} \hspace{-4pt} \dt \nonumber\\
               & \quad +  \sum_{k\in\Kd}\ \intdkmj v(t) \lt \int_t^\tk \lt \int_0^y w(x)\dx \rt^{1-q'} \hspace{-4pt} U^{\frac{p-q}{1-q}}(t,y) \fii^{pq'-p-1}(y)\fii'(y)\dy \rt^{-\frac r{q'}} \hspace{-4pt} \dt \nonumber\\
               & \le      \sum_{k\in\Kj}\ \intdkmj v(t)\dt\ \lt\ \intdkmj \lt \int_0^y w(x)\dx \rt^{1-q'} \hspace{-4pt} U^{\frac{p-q}{1-q}}(\tkmj,y) \fii^{pq'-p-1}(y)\fii'(y)\dy \rt^{-\frac r{q'}} \nonumber\\
               & \quad +  \sum_{k\in\Kd}\ \intdkmj v(t)\Up(t,\tk)\dt\ \lt \int_t^\tk \lt \int_0^y w(x)\dx \rt^{1-q'} \hspace{-4pt} \fii^{pq'-p-1}(y)\fii'(y)\dy \rt^{-\frac r{q'}} \nonumber\\
               & \le      \sum_{k\in\Kj}\ \int_0^\tk v(t)\dt\ \lt\ \intdkmj \lt \int_0^y w(x)\dx \rt^{1-q'} \hspace{-4pt} U^{\frac{p-q}{1-q}}(\tkmj,y) \fii^{pq'-p-1}(y)\fii'(y)\dy \rt^{-\frac r{q'}} \nonumber\\
               & \quad +  \sum_{k\in\Kd}\ \fii^p(\tk)\dt\ \lt \int_t^\tk \lt \int_0^y w(x)\dx \rt^{1-q'} \hspace{-4pt} \fii^{pq'-p-1}(y)\fii'(y)\dy \rt^{-\frac r{q'}} \nonumber\\
               & \lesssim \sum_{k\in\Kj}\ \int_0^\tkmj v(t)\dt\ \lt\ \intdkmj \lt \int_0^y w(x)\dx \rt^{1-q'} \hspace{-4pt} U^{\frac{p-q}{1-q}}(\tkmj,y) \fii^{pq'-p-1}(y)\fii'(y)\dy \rt^{-\frac r{q'}} \label{96}\\
               & \quad +  \sum_{k\in\Kd}\ \fii^p(\tkmj)\dt\ \lt \int_t^\tk \lt \int_0^y w(x)\dx \rt^{1-q'} \hspace{-4pt} \fii^{pq'-p-1}(y)\fii'(y)\dy \rt^{-\frac r{q'}} \nonumber\\
               & \lesssim \sum_{k\in\K}\ \lt\ \intdkmj \lt \int_0^y w(x)\dx \rt^{1-q'} \hspace{-4pt} \fii^{q'-1}(y)\fii'(y)\dy \rt^{-\frac r{q'}} \nonumber \\
               & \lesssim \sum_{k\in\K}\ \lt\ \intdkmj \lt\ \int_\tkmj^y w(x)\dx \rt^{-q'} \!\!\! w(y) \fii^{q'}(y)\dy\rt^{-\frac r{q'}} \!\!\!. \label{97}
      \end{align}
    Step \eqref{96} is based on \eqref{4} and \eqref{5}. In \eqref{97} we used \eqref{95}. Let us continue with $\B{18}$.
      \begin{align}
        \B{18} & \le      \sum_{k\in\K-1}\ \fii^p(\tk) \lt\, \int_\tk^\infty \lt \int_0^y w(x)\dx \rt^{1-q'} \hspace{-4pt} \fii^{pq'-p-1}(y)\fii'(y)\dy \rt^{-\frac r{q'}} \nonumber\\   
               & \lesssim \sum_{k\in\K-1}\ \fii^p(\tk) \lt\ \intdk \lt \int_0^y w(x)\dx \rt^{1-q'} \hspace{-4pt} \fii^{pq'-p-1}(y)\fii'(y)\dy \rt^{-\frac r{q'}} \label{98}\\
               & \le      \sum_{k\in\K-1} \lt\ \intdk \lt \int_0^y w(x)\dx \rt^{1-q'} \hspace{-4pt} \fii^{q'-1}(y)\fii'(y)\dy \rt^{-\frac r{q'}} \nonumber\\
               & \lesssim \sum_{k\in\K}\lt\ \intdkmj \lt\ \int_\tkmj^y w(x)\dx \rt^{-q'} \!\!\! w(y) \fii^{q'}(y)\dy \rt^{-\frac r{q'}}\!\!\!\!. \label{99}
      \end{align}
    Inequality \eqref{98} is a~consequence of Proposition \ref{20} with \eqref{3}, and in \eqref{99} one uses \eqref{95}. We proceed as follows.
      \Bdef{20}\Bdef{21}
      \begin{align*}
        \B{19} & =        \sum_{k\in\K-1}\ \intdkmj v(t)\dt\ \lt \sum_{j=k}^{K-1} \intdj \lt \int_0^y w(x)\dx \rt^{1-q'} \hspace{-4pt} U^{\frac{p-q}{1-q}}(\tk,y) \fii^{pq'-p-1}(y)\fii'(y)\dy \rt^{-\frac r{q'}} \\
               & \lesssim \sum_{k\in\K-2}\ \intdkmj v(t)\dt\ \lt \sum_{j=k+1}^{K-1} U^{\frac{p-q}{1-q}}(\tk,\tj) \intdj \lt \int_0^y w(x)\dx \rt^{1-q'} \hspace{-4pt} \fii^{pq'-p-1}(y)\fii'(y)\dy \rt^{-\frac r{q'}} \\
               & \quad +  \sum_{k\in\K-1}\ \intdkmj v(t)\dt\ \lt \sum_{j=k}^{K-1} \intdj \lt \int_0^y w(x)\dx \rt^{1-q'} \hspace{-4pt} U^{\frac{p-q}{1-q}}(\tj,y) \fii^{pq'-p-1}(y)\fii'(y)\dy \rt^{-\frac r{q'}} \\
               & =: \B{20} +\B{21}.
      \end{align*}
    Next, we have
      \begin{align}
        \B{20} & \le       \sum_{k\in\K-2} 2^{-pk} \int_0^\tk v(s)\ds  \label{100}\\
               & \qquad \times \lt \sum_{j=k+1}^{K-1} \sum_{i=k}^{j-1} \, 2^{\frac{p-q}{1-q}(i+1)} U^{\frac{p-q}{1-q}}(\Delta_i) \intdj  \lt \int_0^y  w(x)\dx \rt^{1-q'} \hspace{-6pt} \fii^{pq'-p-1}(y)\fii'(y)\dy \rt^{\!-\frac r{q'}} \nonumber\\
               & =         \sum_{k\in\K-2} 2^{-pk} \int_0^\tk v(s)\ds \nonumber\\
               & \qquad \times \lt \sum_{i=k}^{K-2} \, 2^{\frac{p-q}{1-q}(i+1)} U^{\frac{p-q}{1-q}}(\Delta_i) \sum_{j=i+1}^{K-1} \intdj \lt \int_0^y w(x)\dx \rt^{1-q'} \hspace{-6pt} \fii^{pq'-p-1}(y)\fii'(y)\dy \rt^{\!-\frac r{q'}}  \nonumber\\
               & =         \sum_{k\in\K-2} 2^{-pk} \int_0^\tk v(s)\ds \nonumber\\
               & \qquad \times \lt \sum_{i=k}^{K-2} \, 2^{\frac{p-q}{1-q}(i+1)} U^{\frac{p-q}{1-q}}(\Delta_i) \int_{t_{i+1}}^\infty \lt \int_0^y w(x)\dx \rt^{1-q'} \hspace{-6pt} \fii^{pq'-p-1}(y)\fii'(y)\dy \rt^{\!-\frac r{q'}} \nonumber\\
               & \lesssim  \sum_{k\in\K-2}\, \int_0^\tk v(s)\ds \ \Up(\dk) \lt\ \int_{t_{k+1}}^\infty \lt \int_0^y w(x)\dx \rt^{1-q'} \!\!\! \fii^{pq'-p-1}(y)\fii'(y)\dy \rt^{\!-\frac r{q'}} \label{101}\\
               & \le       \sum_{k\in\K-2}\, \fii^p(t_{k+1}) \lt\ \int_{t_{k+1}}^\infty \lt \int_0^y w(x)\dx \rt^{1-q'} \!\!\! \fii^{pq'-p-1}(y)\fii'(y)\dy \rt^{\!-\frac r{q'}} \nonumber\\
               & \le       \sum_{k\in\K-2}\, \fii^p(t_{k+1}) \lt\ \int_{\Delta_{k+1}} \lt \int_0^y w(x)\dx \rt^{1-q'} \!\!\! \fii^{pq'-p-1}(y)\fii'(y)\dy \rt^{\!-\frac r{q'}} \label{102}\\
               & \lesssim  \sum_{k\in\K}\, \lt\ \int_{\Delta_{k-1}} \lt \int_0^y w(x)\dx \rt^{1-q'} \!\!\! \fii^{q'-1}(y)\fii'(y)\dy \rt^{\!-\frac r{q'}} \nonumber\\
               & \lesssim  \sum_{k\in\K}\, \lt\ \intdkmj \lt\ \int_\tkmj^y w(x)\dx \rt^{-q'} \!\!\! w(y) \fii^{q'}(y)\dy \rt^{\!-\frac r{q'}}. \label{103}
      \end{align}
    Step \eqref{100} follows from \eqref{81} here $p$ is replaced by $\frac{p-q}{1-q}$. In estimate \eqref{101} we used Proposition \ref{20}, considering also \eqref{104}. Proposition \ref{20} together with \eqref{3} also gives \eqref{102}. In the last step \eqref{103} we used \eqref{99}. Let us now estimate the term $\B{21}$.
      \begin{align}
        \B{21} & \lesssim \sum_{k\in\K-1}\ \intdkmj v(t)\dt\ \lt\ \intdk \lt \int_0^y w(x)\dx \rt^{1-q'} \hspace{-4pt} U^{\frac{p-q}{1-q}}(\tk,y) \fii^{pq'-p-1}(y)\fii'(y)\dy \rt^{-\frac r{q'}} \label{105}\\
               & \lesssim  \sum_{k\in\K}\, \lt\ \int_{\Delta_{k-1}} \lt \int_0^y w(x)\dx \rt^{1-q'} \!\!\! \fii^{q'-1}(y)\fii'(y)\dy \rt^{\!-\frac r{q'}} \nonumber\\
               & \lesssim  \sum_{k\in\K}\, \lt\ \intdkmj \lt\ \int_\tkmj^y w(x)\dx \rt^{-q'} \!\!\! w(y) \fii^{q'}(y)\dy \rt^{\!-\frac r{q'}}. \label{106}
      \end{align}
    As usual, \eqref{105} follows from Proposition \ref{20} with \eqref{2}, and \eqref{106} follows from \eqref{99}. Combining the obtained estimates, we finally get
      \[
        A_6 \lesssim \lt \sum_{k\in\K}\, \lt\ \intdkmj \lt\ \int_\tkmj^y w(x)\dx \rt^{-q'} \!\!\! w(y) \fii^{q'}(y)\dy \rt^{\!-\frac r{q'}} \rt^\jr.
      \]
    By saturation of the discrete H\"older inequality (Proposition \ref{40}), there exists a~nonnegative sequence $\{d_k\}_{k\in\K}$ such that $\sumk d_k^p=1$ and 
      \begin{align}
        &\lt \sum_{k\in\K}\ \lt\ \intdkmj \lt\ \int_\tkmj^y w(x)\dx \rt^{-q'} \!\!\! w(y) \fii^{q'}(y)\dy \rt^{\!-\frac r{q'}}\rt^\jr \nonumber\\
        & \quad \lesssim \lt \sum_{k\in\K}\ d_k^q \lt\ \intdkmj \lt\ \int_\tkmj^y w(x)\dx \rt^{-q'} \!\!\! w(y) \fii^{q'}(y)\dy \rt^{\!-\frac q{q'}} \rt^\jq. \label{107}
      \end{align}
    Considering the functions $\psi_k$ defined previously, put $\psi:=\sumk d_k\psi_k$. Then
      \begin{align}
        & \lt \sum_{k\in\K}\ \lt\ \intdkmj \lt\ \int_\tkmj^y w(x)\dx \rt^{-q'} \!\!\! w(y) \fii^{q'}(y)\dy \rt^{\!-\frac r{q'}}\rt^\jr \nonumber\\
        & \lesssim \lt \sum_{k\in\K}\ d_k^q \lt\ \intdkmj \lt\ \int_\tkmj^y w(x)\dx \rt^{-q'} \!\!\! w(y) \fii^{q'}(y)\dy \rt^{\!-\frac q{q'}} \rt^\jq \label{108} \\
        & \lesssim \lt \sum_{k\in\K}\ d_k^q \intdkmj \lt \int_t^\tk \psi_k(y)\dy \rt^{q} w(t)\dt \rt^{\jq} \label{109} \\
        & =        \lt \sum_{k\in\K}\ \intdkmj \lt \int_t^\tk \psi(y)\dy \rt^{q} w(t)\dt \rt^{\jq} \nonumber \\
        & =        \lt \intnn \lt \int_t^\infty \psi(y)\dy \rt^{q} w(t)\dt \rt^{\jq} \nonumber \\
        & \le C    \lt \intnn v(t)\lt\int_t^\infty u(s) \int_s^\infty \psi(y)\dy\ds \rt^p \dt \rt^\jp \label{110}\\
        & \approx C\lt \sum_{k\in\K} \lt\ \intdkmj \fii(y)\psi(y) \dy \rt^p \rt^\jp \label{111}\\
        & = C      \lt \sum_{k\in\K}\, d_k^p \lt\ \intdkmj \fii(y)\psi_k(y) \dy \rt^p \rt^\jp \nonumber\\
        & = C. \nonumber
      \end{align}
    Inequality \eqref{108} is identical with \eqref{106}, in \eqref{109} one uses \eqref{74}, inequality \eqref{110} follows from \eqref{34}, and \eqref{111} is granted by Theorem \ref{18}. We have shown
      \[
        A_6 \lesssim \lt \sum_{k\in\K}\, \lt\ \intdkmj \lt\ \int_\tkmj^y w(x)\dx \rt^{-q'} \!\!\! w(y) \fii^{q'}(y)\dy \rt^{\!-\frac r{q'}} \rt^\jr \le C
      \]
    and the proof is finished.  
\end{proof}

As it is common when dealing with embeddings of rearrangement-invariant spaces (cf.~\cite{CPSS,GP,GS}), there exist more equivalent conditions characterizing the embedding $\CLj\hra \Lambda^q(w)$. The conditions presented in Theorem \ref{32} were chosen in particular because the weight $w$ appears only at a~single point (in the term $\int_0^t w(s)\ds$) in each of them. Such form is favorable considering the intended use of the results to describe the associated space of $\CL$ in Section 5. In case of interest, the reader may derive other equivalent conditions, possibly with the help of the discrete conditions which appear in the proof of Theorem \ref{32}.


Let us complete the work by extending the results to the case of $\CL$ with a~general positive parameter $m$, using Theorem \ref{32} and Proposition \ref{33}. 

\begin{cor}\label{113}
  Let $m,p,q\in(0,\infty)$, let $(u,v)$ be an~admissible pair of weights with respect to $(m,p)$ and let $w$ be a weight. 
  
  {\rm(i)} 
        Let $0<m\le q<\infty$ and $0<p\le q$. Then \eqref{30} holds for all $f\in\M$ with a~constant $C>0$ independent of $f$ if and only if 
          \[
            A_7 := \sup_{t\in(0,\infty)} \lt\int_0^t w(s)\ds \rt^\jq \lt\int_0^t v(s)\Upm(s,t)\ds \rt^{-\frac1p} <\infty.
          \]
        Moreover, the optimal constant $C$ in \eqref{30} satisfies $C\approx A_7$.
  
  {\rm(ii)} 
        Let $0<m\le q <p<\infty$. Then \eqref{30} holds for all $f\in\M$ with a~constant $C>0$ independent of $f$ if and only if
          \[ 
            A_8 := \lt \intnn v(t) \sup_{y\in(t,\infty)} \frac{\Upm(t,y) \lt \int_0^y w(s)\ds \rt^\frac{p}{p-q} }{ \lt \int_0^y v(x)\Upm(x,y)\dx \rt^{\frac{p}{p-q}}} \dt \rt^\frac{p-q}{pq} < \infty.
          \]
        Moreover, the optimal constant $C$ in \eqref{30} satisfies $C\approx A_8$.
  
  {\rm(iii)} 
        Let $0<p\le q<m<\infty$. Then \eqref{30} holds for all $f\in\M$ with a~constant $C>0$ independent of $f$ if and only if 
          \[ 
            A_9 := \sup_{t>0} \lt \int_0^t v(s)\ds  \int_t^\infty \frac{ \lt \int_0^y w(x)\dx \rt^{\frac m{m-q}} \! \Upm(t,y) u(y) \int_0^y v(s) U^{\frac pm-1}(s,y)\ds}{ \lt \int_0^y v(x)\Upm(x,y)\dx \rt^{2+\frac{mq}{p(m-q)}}} \dy \rt^{\frac{m-q}{mq}} \!\! < \infty,
          \]
          \[ 
            A_{10} := \sup_{t>0} \lt \int_0^t v(s)\Upm(s,t)\ds  \int_t^\infty \frac{ \lt \int_0^y w(x)\dx \rt^{\frac m{m-q}} \!u(y) \int_0^y v(s) U^{\frac pm-1}(s,y) \ds}{ \lt \int_0^y v(x)\Upm(x,y)\dx \rt^{2+\frac{mq}{p(m-q)}}} \dy \rt^{\frac{m-q}{mq}} \!\!\!\! < \infty
          \]
        and
          \[
            A_{11} := \lt \intnn w(t)\dt \rt^\jq \lt\int_0^\infty v(s)\Upm(s,\infty)\ds \rt^{-\frac1p} <\infty.
          \]
        Moreover, the optimal constant $C$ in \eqref{34} satisfies $C\approx A_9 + A_{10} + A_{11}$.
  
  {\rm(iv)} 
        Let $0<q<m<\infty$ and $q<p<\infty$. Then \eqref{30} holds for all $f\in\M$ with a~constant $C>0$ independent of $f$ if and only if $A_{11}<\infty$ and
          \[
            A_{12} := \lt \intnn v(t) \lt \int_t^\infty \frac{ \lt \int_0^y w(x)\dx \rt^{\frac m{m-q}} U^\frac{p-q}{m-q}(t,y) u(y) \int_0^y v(s) U^{\frac pm-1}(s,y)\!\ds}{ \lt \int_0^y v(x)\Upm(x,y)\dx \rt^{1+\frac{m}{m-q}}} \dy \rt^{\frac{p(m-q)}{m(p-q)}} \hspace{-6pt} \dt \rt^\frac{p-q}{pq} \!\! < \infty,
          \]  
        Moreover, the optimal constant $C$ in \eqref{30} satisfies $C\approx A_{11} + A_{12}$.
\end{cor}

\section{Associate space to $C\!L$}

Finally, it is possible to give an~explicit description of the associate space $(\CL)'$ which is generated by the functional
  \[
    \|g\|_{\lt\CL\rt'} := \sup \left\{ \intnn \f(t)\g(t)\dt;\ \|f\|_{\CL}\le 1 \right\},
  \]
defined for any $g\in\M$. As it was mentioned in the introduction, the value of $\|g\|_{\lt\CL\rt'}$ is equal to the optimal constant of the embedding $\CL\hra \Lambda^1(\g)$, hence the results of the previous chapter may be directly applied here.

Once again, the reader should be reminded of the use of convention \eqref{116} in the formulas. Besides that, the notation $\|\cdot\|_1$ is used for the norm in the Lebesgue space $L^1$.
  
\begin{thm}\label{114}
  Let $m,p\in(0,\infty)$, let $(u,v)$ be an~admissible pair of weights with respect to $(m,p)$. Let $g\in\M$.
  
  {\rm(i)} 
        Let $0<m\le 1$ and $0<p \le 1$. Then
          \[  
            \|g\|_{\lt\CL\rt'} \approx \sup_{t>0}\ \gg(t)\, t \lt \int_0^t v(s) \Upm(s,t) \ds \rt^\mjp.
          \]
  
  {\rm(ii)} 
        Let $0<m\le 1<p<\infty$. Then
          \[  
            \|g\|_{\lt\CL\rt'} \approx \lt \intnn v(t) \sup_{y\in(t,\infty)} \frac{\Upm(t,y) (\gg(y))^{p'} y^{p'} }{ \lt \int_0^y v(s) \Upm(s,y) \ds \rt^{p'}} \dy \rt^\frac1{p'}.
          \]
  
  {\rm(iii)} 
        Let $0<p\le 1<m<\infty$. Then
          \begin{align*}  
            & \|g\|_{\lt\CL\rt'} \\
                               & \approx \|g\|_1 \lt \intnn v(s) \Upm(s,\infty) \ds \rt^\mjp \\
                               & \quad + \sup_{t>0} \lt \int_0^t v(s)\ds  \int_t^\infty \frac{ (\gg(y))^{m'} y^{m'} \Upm(t,y) u(y) \int_0^y v(s) U^{\frac pm-1}(s,y)\ds}{ \lt \int_0^y v(s) \Upm(s,y) \ds \rt^{2+\frac{m'}p}} \dy \rt^{\frac1{m'}} \\
                               & \quad + \sup_{t>0} \lt \int_0^t v(s) \Upm(s,t) \ds \int_t^\infty \frac{ (\gg(y))^{m'} y^{m'} \!u(y) \int_0^y v(s) U^{\frac pm-1}(s,y) \ds}{ \lt \int_0^y v(s) \Upm(s,y) \ds \rt^{2+\frac{m'}p} } \dy \rt^{\frac1{m'}} \!\!\!\! .
          \end{align*}
  
  {\rm(iv)} 
        Let $1<m<\infty$ and $1<p<\infty$. Then 
          \begin{align*}  
            & \|g\|_{\lt\CL\rt'} \\
                               & \approx \|g\|_1 \lt \intnn v(s) \Upm(s,\infty) \ds \rt^\mjp \\
                               & \quad + \lt \intnn v(t) \lt \int_t^\infty \frac{ (\gg(y))^{m'} y^{m'} U^\frac{p-1}{m-1}(t,y) u(y) \int_0^y v(s) U^{\frac pm-1}(s,y) \ds}{ \lt \int_0^y v(s) \Upm(s,y) \! \ds \rt^{1+m'}} \dy \rt^\frac{p(m-1)}{m(p-1)} \hspace{-8pt} \dt\rt^{\frac1{p'}} \!\!\!.
          \end{align*}  
\end{thm}

\end{document}